\documentclass[reqno,12pt]{amsart}      
\usepackage{amsfonts}
\usepackage{colordvi}
\usepackage{amsmath,amssymb,amsthm} 
\usepackage{mathrsfs}
\usepackage{graphicx}
\usepackage[compatible]{algpseudocode}
\usepackage{algorithm}

\theoremstyle{plain}      
\newtheorem{theorem}{Theorem}[section]     
     
\newtheorem{lemma}[theorem]{Lemma}

\theoremstyle{remark}      
\newtheorem{example}[theorem]{Example} 
\newtheorem{remark}[theorem]{Remark} 
     
\newtheorem{definition}[theorem]{Definition}     

\newcounter{FNC}[page]
\def\fauxfootnote#1{{\addtocounter{FNC}{2}\Magenta{$^\fnsymbol{FNC}$}%
     \let\thefootnote\relax\footnotetext{\Magenta{$^\fnsymbol{FNC}$#1}}}}

\newcommand{\TEN}{\fontsize{10}{12.2}\selectfont}
\newcommand{\defcolor}[1]{\Blue{#1}}
\newcommand{\demph}[1]{\defcolor{{\sl #1}}}
\def\Color#1#2{#2}

\newcommand{\RP}{\Red{\,\raisebox{-5pt}{\rule{0pt}{17pt}}}}

\def\Nred#1{\Color{0 0.6 0.6 0}{#1}}
\newcommand{\cbd}{{\Black{\bullet}\Nred{\bullet}}}

\newcommand{\Id}{\mbox{\it I}}

\newcommand{\Gr}{\mbox{\rm Gr}}

\renewcommand{\P}{{\mathbb P}}
\newcommand{\C}{{\mathbb C}}

\newcommand{\calM}{{\mathcal M}}
\newcommand{\calT}{{\mathcal T}}
\newcommand{\calX}{{\mathcal X}}
\newcommand{\calY}{{\mathcal Y}}

\newcommand{\bff}{{\bf f}}
\newcommand{\bfh}{{\bf h}}
\newcommand{\bfm}{{\bf m}}

\DeclareMathOperator{\rank}{rank}
\DeclareMathOperator{\Mat}{Mat}

\makeatletter
\newenvironment{breakablealgorithm}
  {
   \begin{center}
     \refstepcounter{algorithm}
     \hrule height.8pt depth0pt \kern2pt
     \renewcommand{\caption}[2][\relax]{
       {\raggedright\textbf{\ALG@name~\thealgorithm} ##2\par}%
       \ifx\relax##1\relax 
         \addcontentsline{loa}{algorithm}{\protect\numberline{\thealgorithm}##2}%
       \else 
         \addcontentsline{loa}{algorithm}{\protect\numberline{\thealgorithm}##1}%
       \fi
       \kern2pt\hrule\kern2pt
     }
  }{
     \kern2pt\hrule\relax
   \end{center}
  }
\makeatother

\newcommand{\g}{\hspace{6pt}}

\newcommand{\MI}{\begin{smallmatrix}1&0&0&0\\0&1&0&0\\0&0&1&0\\0&0&0&1\end{smallmatrix}}
\newcommand{\MII}{\begin{smallmatrix}1&0&\g 0&0\\0&1&\g 0&0\\0&0&\g 1&1\\0&0&-1&0\end{smallmatrix}}
\newcommand{\MIII}{\begin{smallmatrix}1&\g 0&0&0\\0&\g 1&1&0\\0&-1&0&1\\0&\g 1&0&0\end{smallmatrix}}
\newcommand{\MIV}{\begin{smallmatrix}\g 1&1&0&0\\-1&0&1&0\\\g 1&0&0&1\\-1&0&0&0\end{smallmatrix}}
\newcommand{\MV}{\begin{smallmatrix}\g 1&1&\g 0&0\\-1&0&\g 1&1\\\g 1&0&-1&0\\-1&0&\g 0&0\end{smallmatrix}}
\newcommand{\MVI}{\begin{smallmatrix}\g 1&\g 1&1&0\\-1&-1&0&1\\\g 1&\g 1&0&0\\-1&\g 0&0&0\end{smallmatrix}}
\newcommand{\MVII}{\begin{smallmatrix}\g 1&\g 1&\g 1&1\\-1&-1&-1&0\\\g 1&\g 1&\g 0&0\\-1&\g 0&\g 0&0\end{smallmatrix}}

\newcommand{\MIt}{\begin{smallmatrix}1&0&\g 0&0\\0&1&\g 0&0\\0&0&\g 1&1\\0&0&-t&0\end{smallmatrix}}
\newcommand{\MIIt}{\begin{smallmatrix}1&\g 0&0&0\\0&\g 1&1&0\\0&-t&0&1\\0&\g t&0&0\end{smallmatrix}}
\newcommand{\MIIIt}{\begin{smallmatrix}\g 1&1&0&0\\-t&0&1&0\\\g t&0&0&1\\-t&0&0&0\end{smallmatrix}}
\newcommand{\MIVt}{\begin{smallmatrix}\g 1&1&0&0\\-1&0&1&1\\\g 1&0&t&0\\-1&0&0&0\end{smallmatrix}}
\newcommand{\MVt}{\begin{smallmatrix}\g 1&\g 1&1&0\\-1&-t&0&1\\\g 1&\g t&0&0\\-1&\g 0&0&0\end{smallmatrix}}
\newcommand{\MVIt}{\begin{smallmatrix}\g 1&\g 1&\g 1&1\\-1&-1&-t&0\\\g 1&\g 1&\g 0&0\\-1&\g 0&\g 0&0\end{smallmatrix}}

\begin{document}     


\title[Numerical Schubert Calculus]
{Numerical Schubert Calculus via the Littlewood-Richardson Homotopy Algorithm} 

\author[Leykin]{Anton Leykin}
\address{School of Mathematics\\
         Georgia Institute of Technology\\
         686 Cherry Street\\
         Atlanta, GA 30332-0160\\
         USA}
\email{leykin@math.gatech.edu}
\urladdr{http://people.math.gatech.edu/~aleykin3/}
\author[Mart\'in del Campo]{Abraham Mart\'in del Campo}
\address{Abraham Mart\'in del Campo\\
         Centro de Investigaci\'on en Matem\'aticas, A.C.\\
         Jalisco S/N, Col. Valenciana\\
         36023, Guanajuato, Gto.  \\
         M\'exico}
\email{abraham.mc@cimat.mx}
\urladdr{http://personal.cimat.mx:8181/~{}abraham.mc/}
\author[Sottile]{Frank Sottile}
\address{Frank Sottile\\
         Department of Mathematics\\
         Texas A\&M University\\
         College Station\\
         Texas \ 77843\\
         USA}
\email{sottile@math.tamu.edu}
\urladdr{www.math.tamu.edu/\~{}sottile}
\author[Vakil]{Ravi Vakil}
\address{Ravi Vakil\\
      Department of Mathematics\\ 
      Stanford University\\ 
      Stanford, CA 94305
       USA}
\email{vakil@math.stanford.edu} 
\urladdr{http://math.stanford.edu/\~{}vakil}
\author[Verschelde]{Jan Verschelde}
\address{Jan Verschelde\\
      Dept of Math, Stat, and CS\\
      University of Illinois at Chicago\\ 
      851 South Morgan (M/C 249)\\ 
      Chicago, IL 60607
      USA} 
\email{jan@math.uic.edu} 
\urladdr{http://www.math.uic.edu/\~{}jan}
\thanks{This project was supported by American Institute of Mathematics through their SQuaREs program.}
\thanks{Work of Leykin supported in part by the National Science Foundation under grant DMS-1719968}
\thanks{Work of Mart\'in del Campo supported in part by CONACyT under grant C\'atedra-1076}
\thanks{Work of Sottile supported in part by the National Science Foundation under grant DMS-1501370}
\thanks{Work of Vakil supported in part by the National Science Foundation under grant DMS-1500334}
\thanks{Work of Verschelde supported in part by the National Science Foundation under grants ACI-1440534 and DMS-1854513}

\begin{abstract}
 We develop the Littlewood-Richardson homotopy algorithm, which uses 
 numerical continuation to compute solutions to Schubert problems on Grassmannians 
 and is based on the geometric Littlewood-Richardson rule.
 One key ingredient of this algorithm is our new optimal formulation of Schubert problems in 
 local Stiefel coordinates as systems of equations.
 Our implementation can solve problem instances with tens of thousands of solutions.
\end{abstract}

\maketitle

The Schubert calculus on the Grassmannian~\cite{KL72} studies
the linear subspaces that have specified positions with 
respect to fixed flags of linear spaces.
This is a rich class of well-understood geometric problems that appear 
in applications such as 
the pole placement problem in linear systems
theory~\cite{By89,EG2002,KBKK,VW04b}
and in information theory~\cite{BCT}.
Schubert problems serve as a laboratory for investigating new phenomena in enumerative geometry, such as possible numbers
of real solutions~\cite{FRSC-Secant,HHS,So_Shap,RSEG,Ver2000} or monodromy/Galois groups~\cite{LS09,MS,SW,GIVIX}. 
While classical algorithms count the number of solutions~\cite{Fu97},
these applications drive a need to compute the actual solutions to Schubert problems. 

General blackbox symbolic and numerical methods for solving systems of polynomial equations do not perform well on large
Schubert problems, as they are not complete intersections. 
\demph{Numerical Schubert calculus} consists of numerical algorithms adapted to the structure of Schubert problems.
A homotopy algorithm is \demph{optimal} when no solution path diverges for generic instances of the problem.
The Pieri homotopy algorithm for solving special Schubert problems~\cite{HSS98} is an optimal algorithm for Schubert calculus. 
That algorithm is based on a proof of Pieri's rule using geometric specializations~\cite{So97}.
It was implemented and refined~\cite{LWW02,Ver2000,VW04b}, and has been used to 
compute feedback laws for linear systems~\cite{VW04b} and to compute
Galois groups of Schubert problems~\cite{LS09}. 
Special Schubert problems can be formulated as imposing simple rank-deficiency on
several matrices of general linear forms.
Specialized algorithms for solving simple rank-deficiency on a matrix with polynomial
entries were recently developed in~\cite{HSDSV}.

The more general Littlewood-Richardson rule was given a proof using geometric specializations 
organized by a combinatorial checkers game~\cite{Va06a,Vak06b}. 
This geometric rule leads to our main contribution, the first general Littlewood-Richardson homotopy algorithm. 
A preliminary study for this was carried out in~\cite{SVV} for some Schubert problems with a handful of solutions.
The present work is far more intricate and the resulting algorithm is
applicable to {\it any} Schubert problem on a Grassmannian.
A novel feature is that in the homotopy, the underlying space and its parametrization change, but the equations do not.
We have implemented the Littlewood-Richardson homotopy algorithm both in the 
{\tt NumericalSchubertCalculus} package of {\tt Macaulay2}~\cite{M2} and in {\tt PHCpack}~\cite{PHCpack}.
Our software is free and open source, available on {\tt github},
and capable of solving problems with tens of thousands of solutions, 
which are currently far out of reach for all other available methods.   

Section~\ref{S:SC} gives background on the Schubert calculus and numerical homotopy continuation.  
This includes a new formulation for Schubert varieties using the fewest possible number of equations.
Section~\ref{S:GLRR} describes the geometric Littlewood-Richardson rule, 
which is the foundation of our algorithm. 
Section~\ref{S:LRH} is the heart of the paper, 
for it describes the Littlewood-Richardson homotopy algorithm in detail.
Section~\ref{S:Examples} gives some examples of what our software can compute.
Details of the implementations will appear in~\cite{NSC_article}.

\section{Schubert Calculus and Homotopy Continuation}\label{S:SC} 

We describe Schubert problems and explain how they may be
represented on a computer with an efficient set of equations.
This is in terms of local Stiefel coordinates and exploits the Pl\"ucker embedding.
We conclude with a discussion on numerical homotopy continuation.
We will fix positive integers $k<n$ throughout.

\subsection{Schubert problems }\label{SS:SchubertProblems}

The \demph{Grassmannian} \defcolor{$\Gr(k,n)$} of $k$-planes in $\C^n$ is a complex manifold of dimension
$k(n{-}k)$.
It has Schubert subvarieties indexed by \demph{brackets}, which are $k$-element subsets $\alpha$
of $\defcolor{[n]}:=\{1,\dotsc,n\}$, written in increasing order $\alpha\colon\alpha_1<\dotsb<\alpha_k$.
Write \demph{$\binom{[n]}{k}$} for the set of all brackets.
A \demph{flag} $F$ is an increasing sequence of linear subspaces,
\[
   F\ \colon\ F_1\ \subset\ F_2\ \subset\ \dotsb\ \subset\ F_n\ =\ \C^n\,,
   \qquad\mbox{\rm with}\quad \dim F_i=i\,.
\]
A bracket $\alpha\in\binom{[n]}{k}$ and a flag $F$ determine a \demph{Schubert variety}, 
\[
   \defcolor{X_\alpha F}\ :=\ \{H\in\Gr(k,n)\;\mid\;
     \dim(H\cap F_{\alpha_i})\geq i\ \mbox{\rm for}\ i=1,\dotsc,k\}\,.
\]
This variety has dimension $\defcolor{|\alpha|}:= \sum_{i=1}^k(\alpha_i{-}i)$ and its
codimension in $\Gr(k,n)$ is $\defcolor{\|\alpha\|}:=k(n{-}k)-|\alpha|$.

The bracket $[3,4,7,8]\in\binom{8}{4}$ 
and a flag $F$ in $\C^8$ determine the Schubert variety,
 \begin{multline}\label{Eq:X3478}
  \qquad  X_{[3,4,7,8]}F\ =\ \{ H\in\Gr(4,8)\mid
     \dim H\cap F_3\geq 1\,,\ 
     \dim H\cap F_4\geq 2\,,\ \\
     \dim H\cap F_7\geq 3\,,\ \mbox{and}\ 
     \dim H\cap F_8\geq 4\}\,.\qquad
 \end{multline}
This subvariety of $\Gr(4,8)$ has dimension $12=(3-1)+(4-2)+(7-3)+(8-4)$ and codimension $4=4\cdot(8-4)-12$.

The geometric problems studied in Schubert calculus are given by lists of brackets $(\alpha^1,\dotsc,\alpha^s)$ and
flags $F^1,\dotsc,F^s$, and involve understanding the set of $k$-planes in the intersection 
 \begin{equation}\label{Eq:Instance_of_Schubert_Problem}
   X_{\alpha^1}F^1\ \cap\ 
   X_{\alpha^2}F^2\ \cap\ \dotsb\ \cap\ 
   X_{\alpha^s}F^s\,.
 \end{equation}
When the flags $F^1,\dotsc,F^s$ are general and the brackets satisfy
$\|\alpha^1\|+\dotsb+\|\alpha^s\|=k(n{-}k)$, this intersection~\eqref{Eq:Instance_of_Schubert_Problem} 
is zero-dimensional and transverse~\cite{Kl74}, and its
number of points, \defcolor{$d(\alpha^1,\dotsc,\alpha^s)$}, does not depend on the flags.
This number may be computed using combinatorial algorithms from the Schubert calculus~\cite{Fu97}. 
A list of brackets $(\alpha^1,\dotsc,\alpha^s)$ satisfying 
$\|\alpha^1\|+\dotsb+\|\alpha^s\|=k(n{-}k)$ is a \demph{Schubert problem}.
An \demph{instance} of that Schubert problem is given by flags $F^1,\dotsc,F^s$, and 
its \demph{solutions} are the points of the intersection~\eqref{Eq:Instance_of_Schubert_Problem}.

The most basic Schubert problem is $(\alpha,\beta)$ where $\|\alpha\|+\|\beta\|=k(n{-}k)$.
An instance is given by two general flags $F,M$.
The intersection $X_\alpha F\cap X_\beta M$ is empty 
unless $\beta_{k+1-i}=n{+}1-\alpha_i$ for $i=1,\dotsc,k$, and in that case it is the singleton,
 \begin{equation}\label{Eq:trivSChProb}
   X_\alpha F \cap X_\beta M\ =\ \Bigl\{\bigoplus_{i=1}^k F_{\alpha_i}\cap M_{n{+}1-\alpha_i}\Bigr\}\,.
 \end{equation}
As the flags $F$ and $M$ are in general position, $F_{\alpha_i}\cap M_{n{+}1-\alpha_i}$
is one-dimensional.

\subsection{Representing Schubert problems on a computer}

To solve a Schubert problem on a computer 
requires that it be formulated as a system of polynomial equations in some coordinates. 
There are several formulations, including global Pl\"ucker coordinates, local Stiefel coordinates, and 
more exotic primal-dual~\cite{HaHS} or lifted~\cite{HS_Sq} coordinates.
An advantage of local Stiefel coordinates is that they involve the fewest
variables.

An ordered basis $\bff_1,\dotsc,\bff_n$ of $\C^n$ forms the columns of an invertible matrix 
in $\C^{n\times n}$
and vice-versa, with the standard basis corresponding to the identity matrix, \defcolor{$\Id$}.
Given such a basis/matrix, we obtain a flag whose $i$-dimensional subspace is the span of 
the columns $\bff_1,\dotsc,\bff_i$. 
Therefore, two matrices $F,F'$ correspond to the same flag if and only if there is an invertible 
upper triangular matrix $T$
such that $F'=FT$.  
We use the same symbol for an invertible matrix and for the corresponding flag.

The \demph{Stiefel manifold} is the set \defcolor{$\calM_{k,n}$} of $n\times k$ matrices of full rank $k$.
Taking column span leads to a map $\phi\colon\calM_{k,n}\twoheadrightarrow\Gr(k,n)$ which is a principal 
$GL_k(\C)$-bundle.
This admits a (discontinuous) section given by putting any matrix in a fiber into reverse column reduced echelon
form. 
The set \defcolor{$\calX_\alpha$} of \demph{echelon} matrices with pivots in rows $\alpha$ is isomorphic to
$\C^{|\alpha|}$. 
Under $\phi$, the set $\calX_\alpha$ is isomorphic to a dense open subset of the Schubert variety
$X_\alpha\Id$.
For example, when $n=6$ and $k=3$,
here are the sets $\calX_\alpha$ for the brackets
$\alpha=[4,5,6]$, $[2,4,6]$, and $[2,3,5]$, respectively, where $x_{ij}$ indicates an indeterminate:
\[
   \left(\begin{matrix}x_{11}&x_{12}&x_{13}\\x_{21}&x_{22}&x_{23}\\x_{31}&x_{32}&x_{33}\\
                       1&0&0\\0&1&0\\0&0&1\end{matrix}\right)
  \qquad
   \left(\begin{matrix}x_{11}&x_{12}&x_{13}\\1&0&0\\0&x_{32}&x_{33}\\
                       0&1&0\\0&0&x_{53}\\0&0&1\end{matrix}\right)
  \qquad
   \left(\begin{matrix}x_{11}&x_{12}&x_{13}\\1&0&0\\0&1&0\\0&0&x_{43}\\
              0&0&1\\0&0&0\end{matrix}\right)
\]      

A set $\calY\subset\calM_{k,n}$ will be called \demph{Stiefel coordinates} for 
a subvariety $Y$ of $\Gr(k,n)$, if there is an invertible matrix
$M$ such that  $\phi(M\calY)$ is dense in $Y$ and the map $\phi\circ M\colon \calY\to Y$ is birational.
Thus $\calX_\alpha$ gives Stiefel coordinates for the Schubert variety
$X_\alpha\Id$ and also for $X_\alpha M$.
This definition allows the mild but useful ambiguity 
that for $M$ invertible, both $\calX_\alpha$ and $M\calX_\alpha$
 are Stiefel coordinates for both $X_\alpha\Id$ and for $X_\alpha M$.

Given a point $H\in\calM_{k,n}$, the condition that the $k$-plane $\phi(H)$ lies in $X_\alpha F$ may be
expressed in terms of the rank of augmented matrices,
 \begin{equation}\label{Eq:rank_conditions}
    \rank \left(\ H\ \mid\ F_{\alpha_i}\ \right)\ \leq\ k{+}\alpha_i{-}i 
     \qquad\mbox{for}\quad i=1,\dotsc,k\,.
 \end{equation}
Equivalently, for each $i=1,\dotsc,k$, all square $(k{+}\alpha_i{-}i{+}1)\times(k{+}\alpha_i{-}i{+}1)$ minors of the matrix 
$( H\mid F_{\alpha_i})$ vanish.
This gives
\[
   \sum_{i=1}^k \binom{n}{k{+}\alpha_i{-}i{+}1}\binom{k{+}\alpha_i}{k{+}\alpha_i{-}i{+}1}
\]
equations, which are polynomials in the entries of $H$ with coefficients depending upon $F$.
There are no minors when $\alpha_i=n{-}k{+}i$, and conditions are redundant 
if $\alpha_k = n$, or when $1{+}\alpha_i=\alpha_{i+1}$.
For example, when $k=4$, $n=8$, and $\alpha=[3,\defcolor{{\bf 4}},7,8]$,  the only 
meaningful condition in the definition~\eqref{Eq:X3478} of $X_{[3,4,7,8]}F$ is $\dim H\cap F_4\geq 2$, 
or equivalently $\rank(H\mid F_4)\leq 6$.
This is given by the vanishing of
the 64 non-maximal $7\times 7$ minors of the 
$8\times 8$ matrix $(H\mid F_4)$.

This discussion shows that we may model the intersection of a subset $Y\subset\Gr(k,n)$  with a collection of
Schubert varieties,  
\[
   Y\ \cap\ X_{\alpha^1} F^1\:\cap\: X_{\alpha^2} F^2\:\cap\: \dotsb\:\cap\: X_{\alpha^s} F^s\,,
\]
by first selecting a set $\calY\subset\calM_{k,n}$ of Stiefel coordinates for $Y$ and then generating the minors
imposing the rank conditions~\eqref{Eq:rank_conditions}, for each pair $(\alpha^i, F^i)$.

The Littlewood-Richardson Homotopy Algorithm (Algorithm~\ref{Alg:LRH} in Section~\ref{SS:allTogether}) takes as input 
two positive integers $k<n$ indicating the Grassmannian $\Gr(k,n)$, brackets $\alpha^1,\dotsc,\alpha^s$ representing a
Schubert problem on $\Gr(k,n)$, and general flags $F^1,\dotsc,F^s$ in $\C^n$.
Given these, it computes all the solutions to the corresponding instance~\eqref{Eq:Instance_of_Schubert_Problem}. 

\begin{theorem}\label{Th:Main}
 For any Schubert problem $(\alpha^1,\dotsc,\alpha^s)$ and general flags $F^1, \dotsc, F^s$, 
 the Littlewood-Richardson Homotopy Algorithm finds all points in the
 intersection~\eqref{Eq:Instance_of_Schubert_Problem}. 
\end{theorem}

The proof of Theorem~\ref{Th:Main} is included in the proof of correctness 
of the Littlewood-Richardson Homotopy Algorithm.

\subsection{Efficient representation of Schubert problems}
We formulate membership of a $4$-plane in  $X_{[3,4,7,8]} F$ in terms of the Stiefel manifold $\calM_{4,8}$.
The condition~\eqref{Eq:rank_conditions}  on augmented matrices 
is $\rank(H \mid F_4)\leq 6$, where the $4$-plane $H$ is the column space of a $8\times 4$ matrix of
indeterminates and we write the (constant) entries of the $8\times 4$ matrix $F$ as $*$s, 
\[
   (H \mid F_4) \ =\  
  \left(\begin{array}{cccl|rccc}
    x_{11} & x_{12} & x_{13} & x_{14}\mbox{\ } &\mbox{\ }* & * & * & * \\
    x_{21} & x_{22} & x_{23} & x_{24} & * & * & * & * \\
    x_{31} & x_{32} & x_{33} & x_{34} & * & * & * & * \\
    x_{41} & x_{42} & x_{43} & x_{44} & * & * & * & * \\
    x_{51} & x_{52} & x_{53} & x_{54} & * & * & * & * \\
    x_{61} & x_{62} & x_{63} & x_{64} & * & * & * & * \\
    x_{71} & x_{72} & x_{73} & x_{74} & * & * & * & * \\
    x_{81} & x_{82} & x_{83} & x_{84} & * & * & * & * 
  \end{array}\right)\ .
\]
The rank condition is given by the vanishing of the $64$ 
non-maximal minors of $(H\mid F_4)$
obtained by deleting one row and one column.
Half of these equations are homogeneous cubics and half are homogeneous quartics.
The ideal of $X_{[3,4,7,8]}F$ in $\calM_{4,8}$ is generated by only 16 cubic minors, but it is not clear
{\it a priori} which 16 suffice. 
We present another formulation of this Schubert variety that involves only 17 linearly independent quartics.

The Pl\"ucker embedding $\Gr(k,n)\hookrightarrow\P(\wedge^k\C^n)$ is induced by the map
$\Mat_{n\times k}(\C)\to \wedge^k\C^n$ given by the $\binom{n}{k}$ maximal minors of a matrix 
$H=(h_{i,j})\in\Mat_{n\times k}(\C)$
\[
    H\ \longmapsto\ \Bigl( p_\alpha(H)\;\mid\; \alpha\in\tbinom{[n]}{k}\Bigr)\ \in\
    {\textstyle \bigwedge^k}\:\C^n\,.
\]
Here, $\defcolor{p_\alpha(H)}:=\det(h_{\alpha_i,j})_{i,j=1}^k$ is
the determinant of the square submatrix consisting of the rows indexed by $\alpha$ in $H$.
These minors $p_\alpha(H)$ are the \demph{Pl\"ucker coordinates} of $H$.
The image is $\Gr(k,n)$ and it is cut out by the quadratic Pl\"ucker relations~\cite[\S9.1, Lemma~1]{Fu97}.

The Schubert variety $X_\alpha I$ is cut out from $\Gr(k,n)$
by a subset of Pl\"ucker coordinates.
Specifically, $H\in X_\alpha I$ if and only if $p_\beta(H)=0$ 
for all $\beta\in\binom{[n]}{k}$ with $\beta\not\leq\alpha$.
This may be seen as follows.
Given a general matrix $H\in X_\alpha I$, the rank of the square submatrix
formed by its rows
$\beta_1,\dotsc,\beta_k$ is $k$ unless $\beta_i<\alpha_i$ for some $i$.
This uses the partial order on  the index set $\binom{[n]}{k}$ of brackets,
defined by $\alpha\defcolor{\leq}\beta\Longleftrightarrow\alpha_i\leq\beta_i$
for $i=1,\dotsc,k$.

\begin{example}
   When $n=8$, $k=4$, and $\alpha=[3,4,7,8]$, there are 17 brackets $\beta$ with $\beta\not\leq\alpha$:
\begin{center}
   ${\displaystyle [5,6,7,8]\,,\ 
   [4,6,7,8]\,,\ 
   [3,6,7,8]\,,\ 
   [4,5,7,8]\,,\ 
   [2,6,7,8]\,,\ 
   [3,5,7,8]\,,\ 
   [4,5,6,8]\,,}$\\ 
   $\displaystyle{ [1,6,7,8]\,,\ 
   [2,5,7,8]\,,\ 
   [3,5,6,8]\,,\ 
   [4,5,6,7]\,,\ 
   [1,5,7,8]\,,\ 
   [2,5,6,8]\,,\ 
   [3,5,6,7]\,,}$\\
   ${\displaystyle 
   [1,5,6,8]\,,\ 
   [2,5,6,7]\,,\ 
   [1,5,6,7]\,.}$\makebox[0.1pt][l]{\hspace{139pt}$\diamond$}
\end{center}
\end{example}

Observe that $H\in X_\alpha F$ if and only if 
$F^{-1}H\in X_\alpha I$ if and only if $p_\beta(F^{-1}H)=0$ for all $\beta\not\leq\alpha$.
Using the Cauchy-Binet formula, we can write
\[
    p_\beta(F^{-1}H)\ =\ \sum_{\gamma\in\binom{[n]}{k}} p_{\beta,\gamma}(F^{-1}) p_\gamma(H)\,,
\]
where $p_{\beta,\gamma}(F^{-1}):= \det( (F^{-1})_{\beta_i,\gamma_j})_{i,j=1}^k$ is the $(\beta,\gamma)$-th entry in
the matrix $\wedge^k(F^{-1})$.
We summarize this discussion with the following theorem.

\begin{theorem}[\bf Efficient equations for $Y\cap X_\alpha F$]\label{Thm:efficient}
Let $\calY$ be Stiefel coordinates for $Y\subset\Gr(k,n)$ and compute the Pl\"ucker vector
$\defcolor{P(\calY)}:=(p_\beta(\calY)\mid \beta\in\binom{[n]}{k})$ for $\calY$.
Compute the rectangular matrix 
$\defcolor{P(\alpha)(F^{-1})}:=(p_{\beta,\gamma}(F^{-1})\mid \beta\not\leq\alpha,\,\gamma\in\binom{[n]}{k})$.
The entries in the matrix-vector product $P(\alpha)(F^{-1})\cdot P(\calY)$ cut out 
$\phi(\calY)\cap X_\alpha F$ from $\phi(\calY)$.
\end{theorem}
\begin{remark}
 This method is even more efficient for the intersections of several Schubert
 varieties, as we only need to compute $P(\calY)$ once. 
 \hfill$\diamond$
\end{remark}

\begin{remark}
 When this improvement was first implemented in our software, it resulted in speedups of several to 60-fold.
 For instance, for $\alpha=[3,4,7,8]$, computing the problem $(\alpha,\alpha,\alpha,\alpha)$ with six solutions  
 went from 20 minutes to 20 seconds.
 It is implemented in symbolic software used to study Galois groups in Schubert calculus~\cite{GIVIX}.
 \hfill$\diamond$
\end{remark}

\subsection{Numerical homotopy continuation}

A numerical homotopy continuation algorithm computes solutions to a system of polynomial equations by following 
known solutions to a different set of equations along a deformation (homotopy) between the two systems
using predictor-corrector methods.

Suppose that we want to compute the solutions to a system
 \begin{equation}\label{Eq:gen_sys}
    f_1(x_1,\dotsc,x_m)\ =\ 
    f_2(x_1,\dotsc,x_m)\ =\ \dotsb \ =\ 
    f_M(x_1,\dotsc,x_m)\ =\ 0
 \end{equation}
of polynomial equations.
A \demph{homotopy} for~\eqref{Eq:gen_sys} is a one-parameter family of equations $\mathcal{H}(x;t)=0$ whose solutions at
$t=0$ are known and whose solutions at $t=1$ include those of~\eqref{Eq:gen_sys}.
Furthermore, restricting $t$ to the interval $[0,1]$ defines paths in $\C^m$ that connect the 
solutions of~\eqref{Eq:gen_sys} from $t=1$ to known solutions at $t=0$.

For such a homotopy, standard predictor-corrector methods 
are used to numerically trace the known solutions at $t=0$ to obtain solutions
to~\eqref{Eq:gen_sys} at $t=1$ (see~\cite{Mor87} for more details).
The homotopy is \demph{optimal} when every solution at $t=0$ is connected 
to a unique solution to~\eqref{Eq:gen_sys} at $t=1$ along a path. 

This procedure may be iterated, connecting one homotopy to another 
to solve~\eqref{Eq:gen_sys} from known solutions
to another system in two or more steps.
The Pieri homotopy is such an optimal homotopy that used up to $k(n{-}k)-2$ 
steps to solve special Schubert problems~\cite{HSS98}.
The Littlewood-Richardson homotopy (Algorithm~\ref{Alg:LRH} 
in Section~\ref{SS:allTogether}) is also an optimal
homotopy which solves more general Schubert problems on Grassmannians.

\section{The Geometric Littlewood-Richardson Rule} \label{S:GLRR}

The Littlewood-Richardson homotopy algorithm is based on 
the geometric Littlewood-Richardson rule~\cite{Va06a}.
It consists of a sequence of degenerations which successively 
transform an intersection $X_\alpha F\cap X_\beta M$ of
Schubert varieties when $F$ and $M$ are general into a union 
of Schubert varieties $X_\gamma F$ where
$\|\gamma\|=\|\alpha\|+\|\beta\|$. 

These degenerations are encoded in the combinatorial checkerboard game,
described in Section~2 of~\cite{Va06a}.
Subsection 2.18 of {\it loc.~cit.}~explains how 
these are combined into a checkerboard tournament that encodes the process of
resolving a given Schubert problem. 
This checkerboard tournament forms the combinatorial backbone of the Littlewood-Richardson homotopy.

The intermediate components of the degenerations of intersections $X_\alpha F\cap X_\beta M$ are called checkerboard
varieties; these are defined in Subsection~\ref{SS:CBV}, where we also describe Stiefel coordinates for them.
Subsection~\ref{SS:checkers} describes the checkerboard game and explains how to combine 
several of them to get a checkerboard tournament.

\subsection{Checkerboard varieties}\label{SS:CBV}

We summarize salient features of~\cite[Sec.~2]{Va06a}.
Given brackets $\alpha$ and $\beta$, 
the geometric Littlewood-Richardson rule is a sequence of $\binom{n}{2}{+}1$ 
families of subvarieties of $\Gr(k,n)$ parameterized by pairs of flags $(F,M)$ in particular relative positions.
The most general family is parameterized by pairs of flags in general position with the fiber over $(F,M)$ being
the intersection of Schubert varieties $X_\alpha F\cap X_\beta M$.
In the least general family $M=F$ and the fiber over $(F,F)$ is a
union of Schubert varieties $X_\gamma F$ where $\|\gamma\|=\|\alpha\|+\|\beta\|$.
In each intermediate family, the pair of flags $(F,M)$ has a fixed non-general relative position 
and each fiber is a union of
certain checkerboard varieties. These families fit together pairwise into $\binom{n}{2}$ families, 
transforming the intersection  $X_\alpha F\cap X_\beta M$ into a union of Schubert varieties.

These $\binom{n}{2}+1$ families have the same base for any two brackets---each consists of all pairs $(F,M)$ of flags
having a fixed relative position encoded by a permutation $\pi$, where
\[
   \dim(M_i\cap F_j)\ =\ \#\{\ell\leq j \mid \pi(\ell)\leq i\}\,. 
\]
We encode the relative position between $F$ and $M$ in a \demph{permutation array}, which is an $n\times n$ array of boxes with
one black checker $\bullet$ in each row and column.
We will refer to a permutation array by the corresponding permutation $\pi$, defined 
by the positions of the black checkers.
For example, the permutation $356421$ (given in one-line notation) corresponds to the following permutation array.
 \begin{equation}\label{Eq:Perm_Array}
    \raisebox{-25pt}{\includegraphics{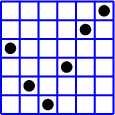}}
 \end{equation}

An ordered basis $\bfm_1,\dotsc,\bfm_n$ for $\C^n$ and a permutation array $\pi$ 
define flags $F$ and $M$ as follows.
Identifying the checker in row $i$ with $\bfm_i$, the $i$-plane $M_i$ is the span of the checkers in the first $i$ rows 
and the $j$-plane $F_j$ is the span of the checkers in the first $j$ columns.  
For example, for the permutation array~\eqref{Eq:Perm_Array},
we have $M_3 = \langle \bfm_1,\bfm_2,\bfm_3 \rangle$, while 
$F_3=\langle \bfm_3,\bfm_5,\bfm_6\rangle$.

A \demph{checkerboard} on a permutation array $\pi$ is a placement $\cbd$ of $k$ red checkers in $\pi$ such that 
the red checkers are in distinct rows and columns, and any subset of $j$ red checkers
has at least $j$ black checkers to its northwest $(\nwarrow)$.
Suppose that $\cbd$ is a checkerboard on a permutation array $\pi$ and $(F,M)$ is a pair of flags having relative
position $\pi$ given by an ordered basis $\bfm_1,\dotsc,\bfm_n$ as above.
For each subset $S$ of red checkers, let \defcolor{$S(F,M)$} be the subspace of $\C^n$ spanned by the black
checkers northwest of $S$.

\begin{definition}\label{De:CheckerboardVariety}
 The \demph{checkerboard variety} $\defcolor{Y_{\cbd}(F,M)}\subset\Gr(k,n)$ consists of all $k$-planes $H$
 such that $\dim H\cap S(F,M)\geq \#S$, for all subsets $S$ of red checkers.
\end{definition}

For the checkerboard $\cbd$ below, the checkerboard variety $Y_{\cbd}(F,M)$  is
\[
    \raisebox{-24pt}{\includegraphics{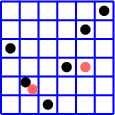}}\qquad
    \begin{array}{rcl}Y_{\cbd}(F,M)&:=&\{H\in\Gr(2,6)\mid \dim H\cap\langle \bfm_3,\bfm_5\rangle\geq 1\,,\vspace{1pt}
      \\ && \mbox{\qquad \ } \dim H\cap\langle \bfm_2,\bfm_3,\bfm_4\rangle\geq 1\,,\vspace{1pt}
      \mbox{ and }\\ && \mbox{\qquad \ } H\subset \langle \bfm_2,\bfm_3,\bfm_4,\bfm_5\rangle\}\,.
    \end{array}
\]
In~\cite{Va06a}, the checkerboard variety $Y_{\cbd}(F,M)$ is called a closed two-flag Schubert variety.
In Lemma~2.6 {\it loc.\ cit.}, an open subset of $Y_{\cbd}(F,M)$ is described as a subset of a tower of
projective bundles.
This is equivalent to the following definition of Stiefel coordinates for a checkerboard variety.

\begin{definition}\label{Def:Stiefel_Coords}
 Order the red checkers from top to bottom.  
 The checkerboard variety $Y_{\cbd}(F,M)$ has Stiefel coordinates given by a set
 $\defcolor{\calY_{\cbd}}=(y_{i,j})$ of reduced echelon matrices as follows. 
The entry $y_{i,j}$ is 0 when the black checker in row $i$ is not northwest of the $j$th red checker, or
if it is northwest and shares its square with a different red checker;
the entry $y_{i,j}$ is a 1 if the $j$th red checker is in row $i$, and otherwise $y_{i,j}$ is an indeterminate.
\end{definition}
 
 The set $\phi(M\calY_{\cbd})$ is dense in the checkerboard variety $Y_{\cbd}(F,M)$. 
 A $k$-plane $H\in\phi(M\calY_{\cbd})$ has a basis $\bfh_1,\dotsc,\bfh_k$ where the vector
 \[
   \bfh_j\ =\ \sum_{i=1}^n y_{i,j}\bfm_i\,,
 \]
corresponds to column $j$ of $M\calY_{\cbd}$.

 By Lemma~\ref{L:redPositions} below, if there is a red checker northwest of red checker $j$, then it lies in the
 square of some (say the $i$th) black checker.
 We may use the column of this northwest red checker to reduce the column of the $j$th red checker in $\calY_{\cbd}$ so that 
 the entry $y_{i,j}$ vanishes.
 Thus this entry must be zero for $\calY_{\cbd}$ to consist of echelon matrices.
 

\begin{example}\label{Ex:CBVar}
Figure~\ref{F:CheckerBoardVariety} shows a checkerboard $\cbd$ and its 
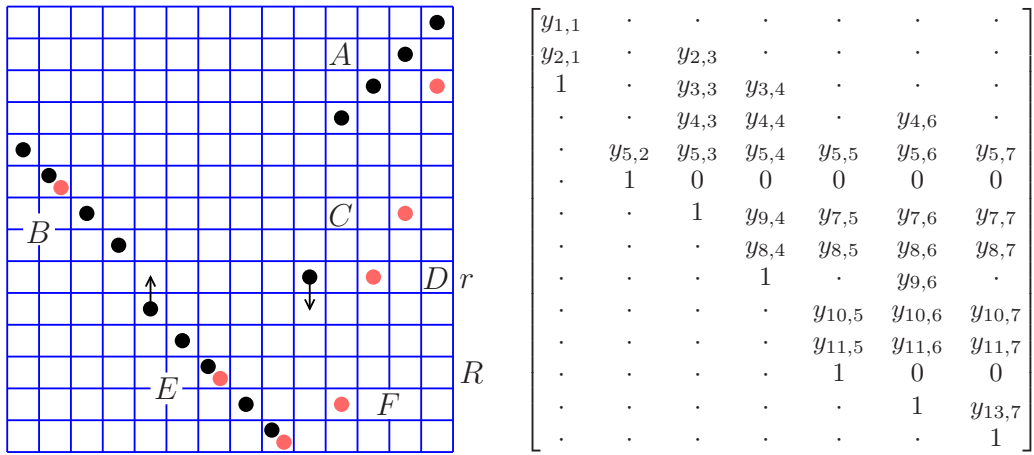
\begin{figure}[htb]
\begin{picture}(405,205)(0,-10)
\put(0,8){
  \begin{picture}(169,169)
   \put(1,1){\includegraphics{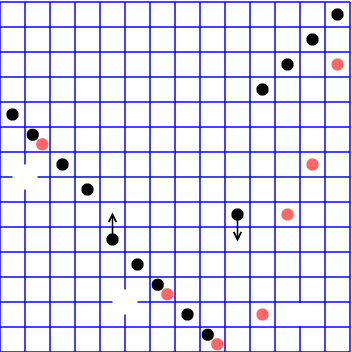}}
   \put(172,27){\Red$R$} \put(172,63){\Red$r$}
   \put(122.5,147.5){$A$}
   \put(  8, 80){$B$}
   \put(122.5, 86.5){$C$}
   \put(158, 62.5){$D$}
   \put( 56.5, 21){$E$}
   \put(140.5, 15){$F$}
  \end{picture}
}
%
%
\put(200,90.5){\TEN
$\left[\begin{matrix}
    y_{1,1} &\cdot &\cdot &\cdot &\cdot &\cdot  &\cdot\\
    y_{2,1} &\cdot &y_{2,3} &\cdot &\cdot &\cdot &\cdot\\ 
    1     &\cdot &y_{3,3} &y_{3,4} &\cdot &\cdot &\cdot\\
    \cdot &\cdot &y_{4,3} &y_{4,4} &\cdot &y_{4,6}&\cdot\\ 
    \cdot &y_{5,2} &y_{5,3} &y_{5,4} &y_{5,5} & y_{5,6}&y_{5,7}\\
    \cdot & 1    & 0    & 0    & 0    & 0    & 0    \\
    \cdot &\cdot & 1    &y_{9,4} &y_{7,5} &y_{7,6} &y_{7,7}\\
    \cdot &\cdot &\cdot &y_{8,4} &y_{8,5}&y_{8,6}  &y_{8,7}\\
    \cdot &\cdot &\cdot & 1    &\cdot &y_{9,6}  &\cdot\\
    \cdot &\cdot &\cdot &\cdot &y_{10,5}&y_{10,6} &y_{10,7}\\
    \cdot &\cdot &\cdot &\cdot &y_{11,5}&y_{11,6}&y_{11,7}\\
    \cdot &\cdot &\cdot &\cdot & 1     & 0    & 0    \\
    \cdot &\cdot &\cdot &\cdot &\cdot  &1     &y_{13,7}\\
    \cdot &\cdot &\cdot &\cdot &\cdot  &\cdot  &1 
\end{matrix}\right]$}
\end{picture}
\caption{Stiefel coordinates corresponding to a checkerboard.}
\label{F:CheckerBoardVariety}
\end{figure}
Stiefel coordinates $\calY_{\cbd}$ when $n=14$ and $k=7$, with permutation array 
$\pi=(6,7,8,9,11,12,13,14,10,5,4,3,2,1)$.
The entries $0$ are forced by the requirement that the matrix be reduced echelon.
The entries $\cdot$ are also $0$ and they indicate that the black checker is not northwest of
the corresponding red checker.
The letters $A,\dotsc,F,r$, and $R$ and the arrows will be explained later.
\hfill $\diamond$
\end{example}

\subsection{The checkerboard game}\label{SS:checkers}

The steps in the geometric Littlewood-Richardson rule, the deformations and degenerations of 
$X_\alpha F\cap X_\beta M$, and of subsequent checkerboard varieties, are all encoded in the 
combinatorial checkerboard game. 
We discuss its salient features, following~\cite[\S\S 2.9--2.19]{Va06a}.

The checkerboard game is a movement of black checkers that encodes the specialization of a pair $(F,M)$ of general flags to the pair $(F,F)$ in special position.
The movement of the black checkers is a bubble sort beginning with the permutation $\omega_0$, where
$\omega_0(i)=n{+}1{-}i$, so that the black checkers will lie on the anti-diagonal.
In the game, the black checkers remain in their respective columns, changing only rows.
The first move interchanges the rows of the lowest (leftmost) two checkers.

For subsequent moves, note that the black checkers of a permutation $\pi$ in mid-sort will be in one of
four regions, illustrated in Figure~\ref{F:risefall}: 
(A) the upper right portion of the anti-diagonal, 
(B) along a diagonal starting in the first column at the row below (A), 
(E) along a diagonal starting one column and two rows after (B), or
there will be a solitary checker (D) in the column between (A) and (E) and in the row between (B) and (E).
If there is no column between the checkers in (A) and those along a diagonal, then consider that diagonal as (E), that (B) is
empty, and the solitary checker (D) is the last checker in (A).
We call the solitary checker (D) the \demph{descending checker} and the top checker in (E) 
the \demph{ascending checker}. 
When $n=4$, there are $7=\binom{4}{2}{+}1$ permutation arrays in the bubble sort.
 \begin{equation}\label{Eq:seven}
  \raisebox{-15pt}{\includegraphics{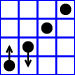}\qquad
   \includegraphics{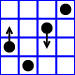}\qquad
   \includegraphics{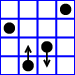}\qquad
   \includegraphics{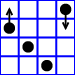}\qquad
   \includegraphics{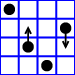}\qquad
   \includegraphics{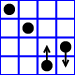}\qquad
   \includegraphics{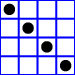}}
 \end{equation}

The subsequent permutation array is obtained by interchanging the rows of the descending and
ascending checkers. 
Call the row of the descending checker the \demph{critical row} and the diagonal (E) the \demph{critical diagonal}. 
See Figure~\ref{F:risefall}.
\begin{figure}[htb]
  \begin{picture}(272,74)(-35,-2)
   \put(80,-2){\includegraphics{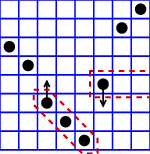}}
   \put(158,29){critical row (D)}
   \put( -35,11){critical diagonal (E)}
   \put( 70,13.5){\vector(1,0){28}}
   \put(154,60){(A)}
   \put( 59,42){(B)}
  \end{picture}
\caption{Critical row and critical diagonal.} 
\label{F:risefall}
\end{figure}
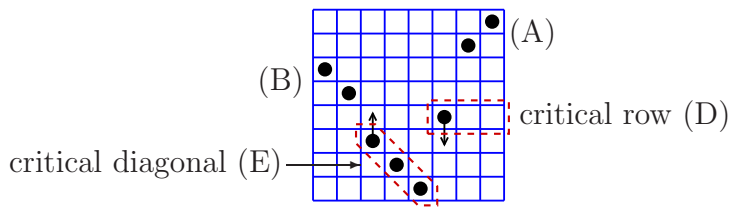

The checkerboard game also constructs a tree with checkerboards as nodes.
This tree is a ranked poset with $\binom{n}{2}+1$ ranks corresponding to the underlying permutation arrays.
Its root encodes the intersection $X_\alpha F\cap X_\beta M$ as a checkerboard for the permutation array $\omega_0$, placing red
checkers in positions $(\beta_{k{+}1{-}i},\alpha_i)$ for $i=1,\dotsc,k$.
When $n=6$ and $k=3$ with $\alpha=[2,4,6]$ and $\beta=[3,4,6]$, we have the following checkerboard $\cbd$ and
Stiefel coordinates $\calY_{\cbd}$ for $X_\alpha F\cap X_\beta M$:
\[
   \raisebox{-30pt}{\includegraphics{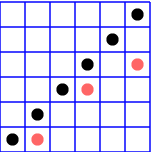}}
   \qquad
  %
  %
   {\TEN\left[\begin{matrix}
      \cdot &\cdot & y_{13}\\
      \cdot &\cdot & y_{23}\\
      \cdot & y_{32} &   1\\
      \cdot &  1   &\cdot\\
      y_{51} &\cdot &\cdot\\
        1   &\cdot &\cdot
   \end{matrix}\right]}\ .
\]
If for some $i$, $\beta_{k+1-i}+\alpha_i<n$, then $X_\alpha F\cap X_\beta M=\emptyset$ and there is no 
checkerboard game.

Each node in this tree has one or two children according to which of nine cases it is in. 
These cases are determined by two questions, each of which has three answers. 
\begin{enumerate}
 \item[]\hspace{-20pt} Where is the top red checker in the critical diagonal $(E)$?
  \begin{itemize}
   \item[$(0)$] In the square of the ascending black checker.
   \item[$(1)$] Elsewhere in the critical diagonal.
   \item[$(2)$] There is no red checker in the critical diagonal.
  \end{itemize}
 \item[]\hspace{-20pt}  Where is the red checker in the critical row $(D)$?
  \begin{itemize}
   \item[$(0)$] In the square of the descending black checker.
   \item[$(1)$] Elsewhere in the critical row.
   \item[$(2)$] There is no red checker in the critical row.
  \end{itemize}
\end{enumerate}

Table~\ref{t:move} shows the movement of the checkers in these nine cases.
The rows correspond to the first question and the columns to the second question.
Only the relevant part of each checkerboard is shown.
\begin{table}[htb]
\caption{Movement of red checkers.} 
\label{t:move}
\begin{center}
 \begin{tabular}{|c|c|c|c|}\hline
 & $0$&$1$&$2$\\\hline
 \raisebox{15pt}{$0$}
 &\raisebox{5.5pt}{\includegraphics{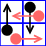}}& \includegraphics{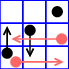}\rule{0pt}{39pt}&
 \raisebox{5.5pt}{\includegraphics{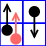}}\\\hline
 \raisebox{19.5pt}{$1$}
 & \raisebox{5.5pt}{\includegraphics{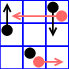}} & \includegraphics{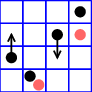}\rule{0pt}{50pt}
   \raisebox{18pt}{\ or\ }  \includegraphics{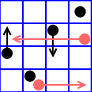} &
   \raisebox{5.5pt}{\includegraphics{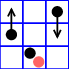}} \\\hline
 \raisebox{15pt}{$2$}&
 \raisebox{5.5pt}{\includegraphics{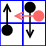}} & \includegraphics{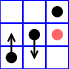}&
 \raisebox{5.5pt}{\includegraphics{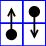}}\rule{0pt}{39pt}\\\hline
\end{tabular}
\end{center}
\end{table}

In case $(1,1)$ there are two possibilities, referred to as \demph{stay} or \demph{swap}, for in one the red
checkers remain in place, while in the other they swap columns. 
The swap occurs only if there are no other
red checkers in the rectangle between the two, called \demph{blockers}.
Figure~\ref{F:blocker} shows a blocker.
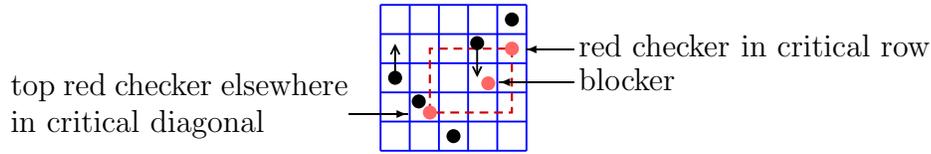
\begin{figure}[htb]
\[
 \begin{picture}(270,47.5)(-40,5)
   \put(70,0){\includegraphics{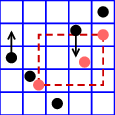}}
   \put(145,36){red checker in critical row}
   \put(143,38.5){\vector(-1,0){18}}

   \put(-70,14){\parbox{128pt}{top red checker elsewhere  in critical diagonal}}
   \put(58,14){\vector(1,0){22}}

   \put(145,23){blocker}\put(143,26){\vector(-1,0){28}}
 \end{picture}
\]
\caption{A blocker.} 
\label{F:blocker}
\end{figure}

A red checker is in region $A$, $B$, or $E$ if both its row and column contain
black checkers in the corresponding region.
Checkers in regions $C$, $D$, or $F$ lie  in the row of some black checker that is in
region $B$, is descending, or is in region $E$, respectively, and they lie in a column of a
black checker in $A$. 
It is helpful to refer to Figure~\ref{F:CheckerBoardVariety}.

\begin{lemma}\label{L:redPositions}
 In a checkerboard $\cbd$, each red checker strictly to the left of the column of the descending checker lies in
 the square of some black checker in region $B$ or $E$.
 The other red checkers are arranged southwest to northeast in regions $F$, $D$, $C$, and $A$.
 In particular, if one red checker is northwest of another, then the first lies in the square of a black checker in
 region $B$ or $E$.
\end{lemma}

\begin{proof}
 This is true in the initial position in the permutation array $\omega_0$, and each move of Table~\ref{t:move}
 preserves this configuration.
\end{proof}

For a permutation $\pi$, let \defcolor{$P_\pi$} be the space of pairs of flags $(F,M)$ in relative position $\pi$.
If $\pi$ follows $\sigma$ in the bubble sort, then in the
space of pairs of flags, $P_\pi$ lies in the closure of $P_\sigma$ and is dense in a
component of $\overline{P_\sigma}\smallsetminus P_\sigma$ so that $\overline{P_\pi}$ is a boundary divisor of
$\overline{P_\sigma}$. 

Suppose that $\cbd'$ is a checkerboard with permutation array $\sigma$ and child checkerboard $\cbd$
with permutation array $\pi$ (or $\cbd$ and $\cbd''$ are its two children in case $(1,1)$ with no blockers).
Let $Y$ be the family over $P_\pi\cup P_\sigma\subset\overline{P_\sigma}$ whose fiber over  $(F,M)\in P_\sigma$ is the
checkerboard variety $Y_{\cbd'}(F,M)$ and over $(F,M)\in P_\pi$ is the checkerboard variety $Y_{\cbd}(F,M)$  
(or $Y_{\cbd}(F,M)\cup Y_{\cbd''}(F,M)$ in case (1,1)).
Then Theorem~2.13 of~\cite{Va06a} states that $Y$ is the closure in 
$(P_\pi\cup P_\sigma)\times\Gr(k,n)$ of its restriction to $P_\sigma$.

At the conclusion of the checkerboard game, all checkers lie along the main diagonal.
For such a checkerboard, the corresponding checkerboard variety is the Schubert variety $X_\gamma F$, where the red
checkers lie in positions $(\gamma_1,\gamma_1),\dotsc,(\gamma_k,\gamma_k)$. 

Figure~\ref{F:checkergame} shows the checkerboard game in the first nontrivial case when $n=4$, $k=2$ and
$\alpha=\beta=[2,4]$.  
It deforms $X_{[2,4]}F\cap X_{[2,4]}M$ into $X_{[1,4]}F\cup X_{[2,3]}F$.
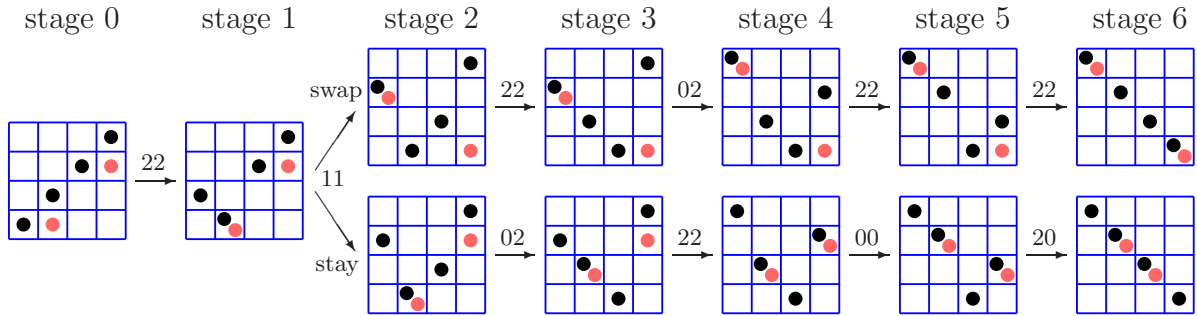
\begin{figure}[htb] 
 \begin{picture}(450,112)(2,46) 
%
%
   \put(  8,154){stage 0} 
   \put( 75,154){stage 1} 
   \put(144,154){stage 2} 
   \put(211,154){stage 3} 
   \put(278,154){stage 4} 
   \put(345,154){stage 5} 
   \put(412,154){stage 6} 
  \put(2,74){\includegraphics{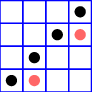}} 
 
  \put(50,96){\vector(1,0){15}} \put(52,100){\scriptsize$22$}
 
  \put(69,74){\includegraphics{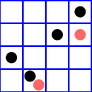}} 
 
  \put(118,100){\vector(2, 3){15}} \put(116,128){\scriptsize swap}
  \put(120, 94){\scriptsize$11$}
  \put(118, 92){\vector(2,-3){15}} \put(118,62){\scriptsize stay}
 
  \put(138,102){\includegraphics{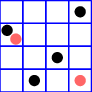}} 
  \put(138, 46){\includegraphics{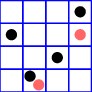}} 

  \put(186,124){\vector(1,0){15}} \put(188,128){\scriptsize$22$}
  \put(186, 68){\vector(1,0){15}} \put(188, 72){\scriptsize$02$}
 
  \put(205,102){\includegraphics{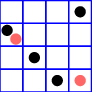}} 
  \put(205, 46){\includegraphics{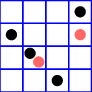}} 
 
  \put(253,124){\vector(1,0){15}} \put(255,128){\scriptsize$02$}
  \put(253, 68){\vector(1,0){15}} \put(255, 72){\scriptsize$22$}
 
  \put(272,102){\includegraphics{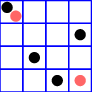}} 
  \put(272, 46){\includegraphics{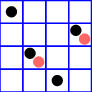}} 
 
  \put(320,124){\vector(1,0){15}} \put(322,128){\scriptsize$22$}
  \put(320, 68){\vector(1,0){15}} \put(322, 72){\scriptsize$00$}
 
  \put(339,102){\includegraphics{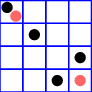}} 
  \put(339, 46){\includegraphics{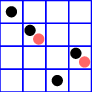}} 
 
  \put(387,124){\vector(1,0){15}} \put(389,128){\scriptsize$22$}
  \put(387, 68){\vector(1,0){15}} \put(389, 72){\scriptsize$20$}
 
  \put(406,102){\includegraphics{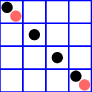}} 
  \put(406, 46){\includegraphics{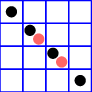}} 
 
  \end{picture}
\caption{Resolving the intersection $X_{[2,4]}F\cap X_{[2,4]}M$.} 
\label{F:checkergame} 
\end{figure} 
The arrows are labeled by the position of the move in Table~\ref{t:move}.
The geometry does not change in the first step, as the $2$-plane $H$ continues to meet both
$M_2=\langle \bfm_1,\bfm_2\rangle$  and $F_2=\langle \bfm_3,\bfm_4\rangle$ in a 1-dimensional subspace.
In the second stage, $H$ continues to meet both $M_2$ and $F_2$, but these now meet in $\langle\bfm_2\rangle$.  
There are two possibilities for $H$ as we are in case $(1,1)$ of Table~\ref{t:move}.
Either $\bfm_2\in H$ (swap) or $H\subset\langle F_2,M_2\rangle=\langle \bfm_1,\bfm_2,\bfm_4\rangle$ (stay).
In subsequent moves the vectors $\bfm_1,\dotsc,\bfm_4$ rearrange themselves.
Three dimensional pictures in~\cite[Figure 4]{SVV} illustrate Figure~\ref{F:checkergame}\footnote{Animations at
  {\tt http://www.math.tamu.edu/\~{}sottile/research/stories/vakil/4lines/1.html}}.

A checkerboard game may have identical nodes.
Since the children of a node depend only on the checkerboard of that node
(and not on the previous history),
we may identify identical nodes, 
obtaining a ranked \demph{checkerboard poset} whose maximal elements (leaves)
are indexed by a subset of those brackets $\gamma$ with $\|\gamma\|=\|\alpha\|+\|\beta\|$.

Suppose that we have a Schubert problem, $(\beta^1,\beta^2,\dotsc,\beta^s)$.
The checkerboard poset for $\beta^1,\beta^2$ has leaves indexed by brackets 
$\alpha$ with  $\|\alpha\|=\|\beta^1\|+\|\beta^2\|$.
For each such $\alpha$, we form the checkerboard poset for $\alpha,\beta^3$ and attach it to the
leaf labeled $\alpha$. 
Identifying identical nodes in this new poset gives a poset whose leaves are indexed by brackets $\gamma$ with 
$\|\gamma\|=\|\beta^1\|+\|\beta^2\|+\|\beta^3\|$.
Repeating this process forms the \demph{checkerboard tournament}, which is a poset having $s{-}2$ levels of
checkerboard posets whose leaves are labeled by brackets $\delta$ with 
$\|\delta\|+\|\beta^s\|=k(n{-}k)$.
We prune this poset, leaving only the single leaf labeled by the sequence 
$\defcolor{(\beta^s)^\vee}:=(n{+}1{-}\beta^s_k,\dotsc,n{+}1{-}\beta^s_1)$.
The number of solutions to the original Schubert problem is the number of saturated chains in this poset from the root to 
the unique leaf, by Corollary~2.17 and the discussion in Subsection~2.18 of~\cite{Va06a}.

\section{The Littlewood-Richardson Homotopy}\label{S:LRH}

We first explain the Littlewood-Richardson homotopy conceptually.
Given a Schubert problem $(\beta^1,\dotsc,\beta^s)$
and flags $F,F^2,\dotsc,F^s$, suppose that we know
all the points of  
 \begin{equation}\label{Eq:smSchubProb}
   X_\gamma F\;\cap\;
   X_{\beta^3} F^3\,\cap\, \dotsb\,\cap\,
   X_{\beta^s} F^s
 \end{equation}
for $\gamma$ any index with $\|\gamma\|=\|\beta^1\|+\|\beta^2\|$.
We use this to find all solutions to the instance of the Schubert problem
 \begin{equation}\label{LgSchubProb}
   X_{\beta^1}F\,\cap\, X_{\beta^2}F^2\;\cap\;
   X_{\beta^3} F^3\,\cap\, \dotsb\,\cap\,
   X_{\beta^s} F^s\,.
 \end{equation}
Formulating membership in $X_{\beta^3} F^3\cap\dotsb\cap X_{\beta^s} F^s$ as a system of polynomial equations,
we use the geometric Littlewood-Richardson rule for $X_{\beta^1}F\cap X_{\beta^2}F^2$ to continue the
points of~\eqref{Eq:smSchubProb} for all $\gamma$ 
back to solutions to the instance~\eqref{LgSchubProb} of the original Schubert problem. 

Similarly, if for some $\ell$, all solutions to instances of Schubert problems of the form
 \begin{equation}\label{Eq:gameLeaf}
   X_\gamma F\,\cap\,
   X_{\beta^\ell} F^\ell\,\cap\, \dotsb\,\cap\,
   X_{\beta^s} F^s
 \end{equation}
are known for all $\gamma$ with $\|\gamma\|+\|\beta^\ell\|+\dotsb+\|\beta^s\|=k(n{-}k)$, then we may find all
solutions to Schubert problems of the form 
 \begin{equation}\label{Eq:gameRoot}
   X_\alpha F\,\cap\,  X_{\beta^{\ell-1}} F^{\ell-1}\;\cap\;
   X_{\beta^\ell} F^\ell\,\cap\, \dotsb\,\cap\,
   X_{\beta^s} F^s \,,
 \end{equation}
for all $\alpha$ with $\|\alpha\|+\|\beta^{\ell-1}\|+\|\beta^\ell\|+\dotsb+\|\beta^s\|=k(n{-}k)$.
Thus starting with the (known) solution~\eqref{Eq:trivSChProb} to $X_{(\beta^s)^\vee} F\cap X_{\beta^s} F^s$, 
after $s{-}2$ iterations of this procedure we obtain all solutions to the original Schubert problem. 

In passing from the Schubert problem~\eqref{Eq:gameLeaf} coming from a leaf of the checkerboard game for the
pair $(\alpha,\beta^{\ell-1})$ to the problem corresponding to its root~\eqref{Eq:gameRoot}, 
we encounter \demph{intermediate Schubert problems} corresponding to nodes $\cbd$ of the checkerboard game.
An instance of the intermediate Schubert problem corresponding to the node $\cbd$ is an intersection 
 \begin{equation}\label{Eq:Intermediate_def}
   Y_{\cbd}(F,M)\ \cap\ 
   X_{\beta^\ell} F^\ell\,\cap\, \dotsb\,\cap\,
   X_{\beta^s} F^s \,.
 \end{equation}
%

Our algorithm requires 1-parameter families of flags to use in each step of
the homotopy.
We also need to specify how the equations are generated, and how the solutions obtained from one 
checkerboard game are passed to the next one in the tournament.

In Subsection~\ref{SS:FlagFamilies} we describe the families of flags underlying each checkerboard game.
In Subsection~\ref{SS:coordinates} we describe the coordinate homotopies, one for each pair of subsequent nodes in a
checkerboard game. 
In Subsection~\ref{SS:allTogether} we explain how these fit together in the Littlewood-Richardson homotopy.

\subsection{Families of flags}\label{SS:FlagFamilies}

The Littlewood-Richardson homotopy uses the degenerations of the geometric Littlewood-Richardson rule along a
sequence of one-parameter families of flags that form a skeleton of the families $P_\sigma$ of
Section~\ref{SS:checkers}. 
This begins with $\binom{n}{2}{+}1$ pairs 
$(F,M)$  of flags in position $\pi$, one pair for each permutation $\pi$ in the bubble sort.
We also select $\binom{n}{2}$ explicit one-parameter families of pairs $(F'(t),M'(t))$ that connect these
flags.
The explicit choices we make here are those made in our software. 
The flags $F$ and $F'(t)$ are fixed to be the standard coordinate flag, so we only need to specify the
flags $M$ and $M'(t)$ for each permutation and family.
These have the following property.
If $M'$ corresponds to the permutation $\sigma$ and $M$ to the next permutation
$\pi$ in the bubble sort, then the family $M'(t)$ connecting them satisfies
 \begin{equation}\label{Eq:connecting}
   M'(0)\ =\ M   \qquad\mbox{and}\qquad   M'(1)\ =\ M'\,,
 \end{equation}
and for all $t\neq 0$, the pair $(F,M'(t))$ has position $\sigma$.

The subspace $F_i$ of $F$ is spanned by the $i$th column of the identity matrix.
At a permutation $\pi$, the flag $M$ is given by an ordered basis $\bfm_1,\dotsc,\bfm_n$ so
that $M_i$ is spanned by $\bfm_1,\dotsc,\bfm_i$ while $F_i$ is spanned by
$\bfm_{\pi(1)},\dotsc,\bfm_{\pi(i)}$, but $\bfm_1,\dotsc,\bfm_n$ is not 
necessarily a permutation of the columns of the identity matrix.
This is illustrated in the second row of Figure~\ref{F:matrices}.

At the leaves of a checkerboard game, $M=F$. We describe the other flags recursively.
Suppose that the flag $M$ corresponds to a permutation $\pi$ in the bubble sort with $\sigma$ the
previous permutation, and let $r$ be the critical row in the sort from $\sigma$ to $\pi$.
Then the flag $M'$ corresponding to $\sigma$ is given by the basis $\bfm'_1,\dotsc,\bfm'_n$, where
%
%
 \begin{equation}\label{Eq:MandMprime}
   \bfm'_i\ =\ \bfm_i\quad\mbox{for }i\neq r,r{+}1\,,
    \qquad
   \bfm'_r\ =\ \bfm_r-\bfm_{r+1}\,,
     \qquad\mbox{and}\qquad
   \bfm'_{r+1}\ =\ \bfm_r\,.
 \end{equation}
For $t\neq 0$, the family $M'(t)$ is given by the basis
$\bfm'_1(t),\dotsc,\bfm'_n(t)$, where  
 \begin{eqnarray}
  \bfm'_i(t)&=& \bfm_i\ =\ \bfm'_i\qquad i\neq r,r{+}1\,,\nonumber\\
  \bfm'_r(t)&=& \bfm_r-t\bfm_{r+1}\ =\ t\bfm'_r + (1-t)\bfm'_{r+1}\,,\qquad\mbox{and}\label{Eq:movingFlag}\\
  \bfm'_{r+1}(t)&=& \bfm_r\ =\ \bfm'_{r+1}\,.\nonumber
 \end{eqnarray}
For $t\neq 0$, we have $\langle\bfm'_r(t),\bfm'_{r+1}(t)\rangle=\langle\bfm_r,\bfm_{r+1}\rangle$.
As $\lim_{t\to 0}M'(t)=M$, we set $M(0):=M$.  
The flag $M$ at the root corresponds to the triangular matrix $(m_{i,j})$,
where 
\[
    m_{i,j}\ =\ \left\{\begin{array}{rcl} 0&\ &\mbox{if } n<j+i\\
                                       (-1)^i&&\mbox{otherwise.}\end{array}\right.
\]

Figure~\ref{F:matrices} shows the permutations, arrays, matrices $M$, and families $M'(t)$ when
$n=4$.
\begin{figure}[htb]
 \begin{picture}(408,130)
  \put( 31,120){$1234$}\put( 87,120){$1243$}\put(143,120){$1342$}\put(199,120){$2341$}
  \put(255,120){$2431$}\put(312,120){$3421$}\put(368,120){$4321$}

  \put( 25,77){\includegraphics{pictures/1234.eps}}  \put( 81,77){\includegraphics{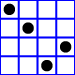}}
  \put(137,77){\includegraphics{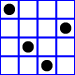}}  \put(193,77){\includegraphics{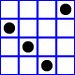}}
  \put(249,77){\includegraphics{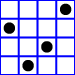}}  \put(305,77){\includegraphics{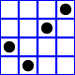}}
  \put(361,77){\includegraphics{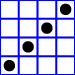}}
  \put(0,48){$M$}    \put(27,48){$\MI$}   \put(80,48){$\MII$}
                   \put(136,48){$\MIII$}  \put(189,48){$\MIV$}
                   \put(243,48){$\MV$}    \put(298,48){$\MVI$}
                  \put(352,48){$\MVII$} 
  \put(0,10){$M'(t)$}\put(50,10){$\MIt$}   \put(105,10){$\MIIt$}
                   \put(163,10){$\MIIIt$}  \put(218,10){$\MIVt$}
                   \put(272,10){$\MVt$}    \put(325,10){$\MVIt$}
 \end{picture}
 \caption{Permutation arrays, matrices $M$, and families of matrices $M(t)$.}
 \label{F:matrices}
\end{figure}

\subsection{Stiefel coordinates and homotopy for checkerboard moves}\label{SS:coordinates}
Suppose that the permutation $\sigma$ is followed by $\pi$ in the bubble sort.
Fix, as in Subsection~\ref{SS:FlagFamilies}, the flags $F$, $M$, $M'$, and $M'(t)$.
Let $\cbd'$ be a checkerboard with permutation array $\sigma$ and suppose
that $\cbd$ is a child checkerboard of $\cbd'$ with permutation array $\pi$.
Then by Theorem 2.3 of~\cite{Va06a} 
the family of checkerboard varieties $Y_{\cbd'}(F,M'(t))$ for $t\neq 0$
extends to a family $Y_{\cbd,\cbd'}(t)$ over $\C$ with $Y_{\cbd}(F,M)$ a component of the special fiber at $t=0$.
(If $\cbd$ is the unique child checkerboard of $\cbd'$, then $Y_{\cbd}(F,M)$ is the special fiber, otherwise 
there is a second component $Y_{\cbd''}(F,M)$ corresponding to the other child $\cbd''$.)

The key construction in the Littlewood-Richardson homotopy is a set of Stiefel coordinates $\calY_{\cbd}(t)$ for
this family, in the following sense.
 \begin{enumerate}
  \item[(i)] $\calY_{\cbd}(0)$ are Stiefel coordinates for $Y_{\cbd}(F,M)$ in that $\phi(M\calY_{\cbd}(0))$ is
         dense in the checkerboard variety $Y_{\cbd}(F,M)=Y_{\cbd}(F,M(0))$. 
 
  \item[(ii)] For $t\neq 0$, we have that $\phi(M\calY_{\cbd}(t))$ is dense in $Y_{\cbd'}(F,M(t))$.
 \end{enumerate}
Thus $\calY_{\cbd}(t)$ gives Stiefel coordinates for the family $Y_{\cbd'}(F,M(t))$, parameterizing
an open subset that meets the component $Y_{\cbd}(F,M)$ of the special fiber.
These coordinates $\calY_{\cbd}(t)$ will be defined below and their properties verified.

\begin{remark}\label{r:twoChildren}
  If $\cbd'$ has another child $\cbd''$, then $\calY_{\cbd''}(t)$ also gives Stiefel coordinates for 
  $Y_{\cbd'}(F,M(t))$ and $\phi(M\calY_{\cbd''}(t))$ meets the component $Y_{\cbd''}(F,M)$ of the special
  fiber.  \hfill $\diamond$
\end{remark}

These Stiefel coordinates $\calY_{\cbd}(t)$ are used to generate a homotopy 
corresponding to the edge $\cbd$--$\cbd'$ in the checkerboard tournament.
We describe this homotopy.

%
\medskip

\begin{breakablealgorithm}
\caption{(Checkerboard Homotopy Algorithm)}\label{Alg:core}
\begin{algorithmic}[1]
Let $(\gamma,\beta^{\ell-1},\beta^\ell,\dotsb,\beta^s)$ be a Schubert problem and 
suppose that $\cbd$--$\cbd'$ is an edge in the checkerboard game 
for $(\gamma,\beta^{\ell-1})$ with $\cbd'$ the parent of $\cbd$. 
\REQUIRE{A solution \defcolor{$y^*$} to the instance of the intermediate problem
%
%
%
 \[  
   Y_{\cbd}(F,M)\ \cap\ X_{\beta^\ell}F^{\ell}\ \cap\ \dotsb\ \cap\ X_{\beta^s}F^s\,
 \]  
  represented as a matrix $(y^*_{i,j})\in\calY_{\cbd}$ such that $y^*=\phi(M(y^*_{i,j}))$.}
  
\ENSURE{The solution \defcolor{$y'$} to the instance of the intermediate problem
 \begin{equation}\label{Eq:target}
     Y_{\cbd'}(F,M')\ \cap\ X_{\beta^\ell}F^{\ell}\ \cap\ \dotsb\ \cap\ X_{\beta^s}F^s\,
 \end{equation}
  connected to $y^*$ by the family $Y_{\cbd,\cbd'}(t)$ for $t\in[0,1]$, which is represented by a matrix 
  $(y'_{i,j})\in\calY_{\cbd'}$ with $y'=\phi(M'(y'_{i,j}))$.}
\STATE{Generate the coordinates $\calY_{\cbd}(t)$ for $Y_{\cbd,\cbd'}(t)$.}
\STATE{The homotopy $\mathcal{H}(y;t)$ is given by the equations  of Theorem~\ref{Thm:efficient} for membership in the
  Schubert varieties $ X_{\beta^\ell}F^{\ell}, \dotsc, X_{\beta^s}F^s$ evaluated on the Stiefel coordinates $M
  \calY_{\cbd}(t)$.} 
\STATE{Use numerical continuation to follow the homotopy $\mathcal{H}(y;t)$ from the the start solution  $(y^*_{i,j})$
  at $t=0$ to a solution $(y^*_{i,j}(1))$ at $t=1$.} 
\STATE{
   Solve the equation 
 \begin{equation}\label{Eq:easyLinAlg}
   M' (\tilde{y}_{i,j})\ =\ M (y^*_{i,j}(1))
 \end{equation}
 for the matrix $(\tilde{y}_{i,j})$.
}
\STATE{Put the solution $(\tilde{y}_{i,j})$ in echelon form to get a point $(y'_{i,j})\in\calY_{\cbd'}$.}
\end{algorithmic}
\end{breakablealgorithm}

\begin{proof}[Proof of correctness]
  The coordinates $\calY_{\cbd}(t)$ satisfy the properties (i) and (ii) above and thus $M\calY_{\cbd}$ gives
  Stiefel coordinates for the family $Y_{\cbd,\cbd'}(t)$.
  It follows that this homotopy computes a point $y'$ in 
  the target intermediate problem~\eqref{Eq:target}.
 
 The arguments in Cases I--III below show that the echelon form of a solution to~\eqref{Eq:easyLinAlg} 
 lies in the Stiefel coordinates $\calY_{\cbd'}$, which completes the proof.
\end{proof}

\begin{remark}\label{r:nochange}
 In passing from $\pi$ to $\sigma$, the black checkers in rows
 $r$ and $r{+}1$ switch rows, 
\[
   \pi\ \colon\ \raisebox{-8pt}{\begin{picture}(44,22)
     \put(0,0){\includegraphics{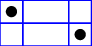}}
     \put(15,5){$\dotsc$}     \put(15,16){$\dotsc$}
    \end{picture}}
    \qquad\mbox{becomes}\qquad
   \sigma\ \colon\ \raisebox{-8pt}{\begin{picture}(44,22)
     \put(0,0){\includegraphics{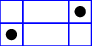}}
     \put(15,5){$\dotsc$}     \put(15,16){$\dotsc$}
    \end{picture}}\ .
\]
 If $M$ is the flag for $\pi$ and $M'$ the flag for $\sigma$, then by~\eqref{Eq:MandMprime}, $\bfm'_{r+1}=\bfm_r$
 and $\bfm'_r=\bfm_r-\bfm_{r+1}$.
 Thus the basis element corresponding to the left moving black checker is unchanged,
 while that corresponding to the right moving black checker is changed,
 but their span is unchanged.
 It follows that if there is no red checker in the critical row $r$, then the geometric
 condition on the $k$-plane is unchanged in the move.
\hfill $\diamond$
\end{remark}

As there are ten different checkerboard moves in Table~\ref{t:move}, there are potentially ten different families
of Stiefel coordinates $\calY_{\cbd}$ for the family $Y_{\cbd,\cbd'}(t)$.
Analyzing their geometry reveals there are only three geometrically distinct cases for the construction of
$\calY_{\cbd}(t)$.
We indicate these cases by their positions in the $3\times 3$ array of Table~\ref{t:move},
\[
    \mbox{I}\ \colon\ 
    \raisebox{-9pt}{\begin{picture}(27,27)
      \put(0,0){\includegraphics{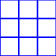}}
      \put(19.5,20){{\small x}}
      \put(19.5,11){{\small x}}      
      \put(19.5, 2){{\small x}}
    \end{picture}}\ ,
     \qquad
    \mbox{II}\ \colon\ 
    \raisebox{-10pt}{\begin{picture}(40,29)
      \put(0,0){\includegraphics{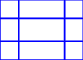}}
      \put(10.5,12){{\small stay}}
      \put(1.5,2){{\small x}}      \put(17,2){{\small x}}
    \end{picture}}\ ,
     \quad\mbox{and }\quad
    \mbox{III}\ \colon\ 
    \raisebox{-10pt}{\begin{picture}(44,29)
      \put(0,0){\includegraphics{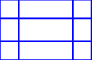}}
      \put(1.5,22){{\small x}}       \put(20,22){{\small x}}
        \put(1.5,12){{\small x}}     \put(10.5,12){{\small swap}}
    \end{picture}}\ ,
\]
and refer to them by the numerals I, II, and III in the sequel.\medskip

\noindent{\bf Case I.}\ 
There is no red checker in the critical row, so 
the geometric condition on the $k$-plane does not change, as noted in Remark~\ref{r:nochange}.
We need only to explain how to transform the coordinates $\calY_{\cbd}$ of a given $k$-plane
into the coordinates $\calY_{\cbd'}$ so that 
 \begin{equation}\label{Eq:I_correct}
   M'\calY_{\cbd'}\ =\ M \calY_{\cbd}\,.
 \end{equation}
({\it cf}.~\eqref{Eq:easyLinAlg}.)
Write $y'_{i,j}$ and $y_{i,j}$ for the entry in row $i$ and column $j$ of $\calY_{\cbd'}$
and $\calY_{\cbd}$ respectively, and let $r$ be the critical row.
If we set
%
%
 \begin{equation}\label{Eq:CaseI}
   y'_{i,j}\ :=\ y_{i,j}\ \ i\neq r,r{+}1\,, \quad
   y'_{r,j}\ :=\ -y_{r+1,j}\,,\quad 
   \mbox{and}\quad
   y'_{r+1,j}\ :=\ y_{r,j}+y_{r+1,j}\,,
 \end{equation}
then~\eqref{Eq:I_correct} is satisfied as
 \begin{eqnarray*}
 y_{r,j}\bfm_r+y_{r+1,j}\bfm_{r+1}&=&
    y_{r,j}\bfm'_{r+1} + y_{r+1,j}(\bfm'_{r+1}-\bfm'_r)\\
    &=& -y_{r+1,j}\bfm'_r \, +\, (y_{r,j}+y_{r+1,j})\bfm'_{r+1}\\
    &=&  y'_{r,j}\bfm'_r \, +\,  y'_{r+1,j}\bfm'_{r+1}\,.
 \end{eqnarray*}
In practice, our software solves the equation~\eqref{Eq:I_correct} for the entries of $\calY_{\cbd'}$.

If there is a red checker in row $r{+}1$ of $\cbd'$, then its column will not be in echelon form in
$\calY_{\cbd'}:$
If its column index  is $j$, then the last two non-zero entries are in rows $r$ and $r{+}1$, 
and they are $y'_{r,j}=-1$ and $y'_{r+1,j}=1+y_{r,j}$.
In this case, we  divide that column by $y'_{r+1,j}$ to put $\calY_{\cbd'}$ into (reduced) echelon form,
as we do in our software.\medskip

%
\noindent{\bf Case II.}\ 
As there is a checker in the critical row, by Remark~\ref{r:nochange}, the geometric condition on the $k$-plane changes
and the Stiefel coordinates $\calY_{\cbd}(t)$ will involve $t$.
We describe them  and then prove they have the properties claimed.
We will write $j\in A, B$ to indicate that the $j$th red checker of $\cbd$ is in region $A$ or in region $B$, and
  the same for the other regions or rows of the checkerboard as defined in Figure~\ref{F:CheckerBoardVariety}.

Let $(y_{i,j})=\calY_{\cbd}$ be the Stiefel coordinates from Definition~\ref{Def:Stiefel_Coords}.
Define $\calY_{\cbd}(t)=(y_{i,j}(t))$, by setting $y_{i,j}(t):=y_{i,j}$ if $i\neq r{+}1$.
When $i=r{+}1$, set $y_{r+1,j}(t) := y_{r+1,j} = 0$ if $j\in E$, and otherwise set 
 \begin{equation}\label{Eq:caseII}
   y_{r+1,j}(t)\ :=\ y_{r+1,j}\ -\ ty_{r,j}\,.
 \end{equation}
Observe that if $j\in A, B,$ or $C$, then its row is above $r$ so that $y_{r+1,j}=y_{r,j} = y_{r+1,j}(t) = 0$.
Note that $y_{r+1,j}(t)$ is non-zero when $j\in F$ or when $j$ lies in row $r$, for when $j$ lies in row $r$,
$y_{r,j}=1$ and $y_{r+1,j}=0$.

\begin{lemma}\label{L:St_family}
 For any $t\neq 0$, $\phi(M\calY_{\cbd}(t))$ is dense in the checkerboard variety $Y_{\cbd'}(F,M'(t))$ and
 $\phi(M\calY_{\cbd}(0))$ is dense in $Y_{\cbd}(F,M'(0))$.
\end{lemma}

\begin{proof}
 When $t=0$, this holds as $\calY_{\cbd}(0)=\calY_{\cbd}$,  $M\calY_{\cbd}$
 gives Stiefel coordinates for $Y_{\cbd}(F,M)$, and $M'(0)=M$.
 For $t\neq 0$, we will show that if we solve the equation $M'(t)\calY_{\cbd'}(t)=M\calY_{\cbd}(t)$ 
 for the $n\times k$ matrix $\calY_{\cbd'}(t)$, 
 then $\calY_{\cbd'}(t)$ for $t\neq 0$ is a curve in $\calY_{\cbd'}$ whose entries are functions of $y_{i,j}$ and $t$.

 Let $\bfh_1(t),\dotsc,\bfh_k(t)$ be the column vectors of $M\calY_{\cbd}(t)$, which span the $k$-plane
 $\phi(M\calY_{\cbd}(t))$. 
 If $j\in E$, then 
\[
   \bfh_j(t)\ =\ \sum_{i\neq r,r+1} y_{i,j}\bfm_i\ +\ y_{r,j}\bfm_r\ +\ 0\cdot\bfm_{r+1}\,,
\]
 as $y_{r+1,j}=0$.
 If $j\not\in E$, then  by~\eqref{Eq:caseII}, 
\[
   \bfh_j(t)\ =\ \sum_{i\neq r,r+1} y_{i,j}\bfm_i\ +\ y_{r,j}\bfm_r\ +\ 
           (y_{r+1,j} - t y_{r,j})\bfm_{r+1}\,.
\]

 Let us express $\bfh_i(t)$ in the basis $\bfm'_1(t),\dotsc,\bfm'_n(t)$.
 If $i\neq r,r{+}1$, then by~\eqref{Eq:movingFlag}, $\bfm'_i(t)=\bfm_i$, and we have  $\bfm'_r(t)=\bfm_r-t\bfm_{r+1}$
 and $\bfm'_{r+1}(t)=\bfm_r$, so that when $t\neq 0$, we have 
\[
    \bfm_{r+1}\ =\ \tfrac{1}{t}(\bfm'_{r+1}(t)\ -\ \bfm'_r(t))\,.
\]
 If $j\in E$, we have 
\[
   \bfh_j(t)\ =\ \sum_{i\neq r,r+1} y_{i,j}\bfm'_i(t)\ +\ y_{r,j}\bfm'_{r+1}(t)\,.
\]
 If $j\not\in E$, then 
\[
   \bfh_j(t)\ =\ \sum_{i\neq r,r+1} y_{i,j}\bfm'_i(t)\ +\ (y_{r,j}-\tfrac{1}{t}y_{r+1,j})\bfm'_r(t)\ +\ 
           \tfrac{1}{t}y_{r+1,j}\bfm'_{r+1}(t)\,.
\]

 Define the Stiefel coordinates $\calY_{\cbd'}(t)=(y'_{i,j}(t))$ for $t\neq 0$ by
\[
   y'_{i,j}(t)\ =\ y_{i,j}\ \quad\mbox{for}\quad i\neq r,r{+}1\,,
\]
 and if $j\in E$, then 
\[
  y'_{r,j}(t)\ =\ 0\ =\ y_{r+1,j}
   \qquad\mbox{and}\qquad
  y'_{r+1,j}(t)\ =\ y_{r,j}\,,
\]
 and if $j\not\in E$, then 
\[
   y'_{r,j}(t)\ =\ y_{r,j}\ -\ \tfrac{1}{t}y_{r+1,j}
   \qquad\mbox{and}\qquad
  y'_{r+1,j}(t)\ =\ \tfrac{1}{t}y_{r+1,j}\,.
\]
 A consequence of these definitions is that for $t\neq 0$, the column vectors of $M'(t)\calY_{\cbd'}(t)$ are equal to
 $\bfh_1(t),\dotsc,\bfh_k(t)$.
 That is, 
\[
   M'(t)\calY_{\cbd'}(t)\ =\  M \calY_{\cbd}(t)\,.
\]
 Note that the entry $y'_{i,j}(t)$ is 0, 1, or an affine polynomial in the $y_{p,q}$ and $\frac{1}{t}$ if and only if
 the corresponding entry in the Stiefel coordinates $\calY_{\cbd'}$ of Definition~\ref{Def:Stiefel_Coords} is 0,
 1, or an indeterminate, respectively.
 This proves the lemma.
\end{proof}

\noindent {\bf Case  III.}\ 
This case is the most subtle.
Let $\cbd$ be a child of $\cbd'$ with the checkerboard move in Case III 
in which two red checkers move columns.
Let $(y_{i,j})$ be the entries in $\calY_{\cbd}$, as given in Definition~\ref{Def:Stiefel_Coords}.
Let \defcolor{$s$} be the index of the red checker in the critical row $r$, and \defcolor{$s{+}1$} the index of
the other moving red checker, which is in row $R\geq r{+}1$.

Figure~\ref{F:ChildBoard} gives an example of $\cbd$ and $\calY_{\cbd}$, which is a child of the checkerboard
\begin{figure}[htb]
\begin{picture}(405,205)(0,-10)
\put(0,8){
  \begin{picture}(169,169)
   \put(1,1){\includegraphics{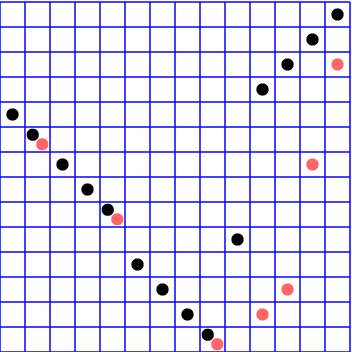}}
   \put(172,27){$R$} \put(172,63){$r$}
  \end{picture}
}
%
%
\put(200,87.5){\TEN
$\left[\begin{matrix}
    y_{1,1} &\cdot &\cdot &\cdot &\cdot &\cdot  &\cdot\\
    y_{2,1} &\cdot &y_{2,3} &\cdot &\cdot &\cdot &\cdot\\ 
    1     &\cdot &y_{3,3} &\cdot &y_{3,5} &\cdot &\cdot\\
    \cdot &\cdot &y_{4,3} &\cdot &y_{4,5} &y_{4,6}&\cdot\\ 
    \cdot &y_{5,2} &y_{5,3} &y_{5,4} &y_{5,5} & y_{5,6}&y_{5,7}\\
    \cdot & 1    & 0     & 0    & 0    & 0    & 0    \\
    \cdot &\cdot & 1    &y_{7,4} &y_{7,5} &y_{7,6} &y_{7,7}\\
    \cdot &\cdot &\cdot &y_{8,4} &y_{8,5}&y_{8,6}  &y_{8,7}\\
    \cdot &\cdot &\cdot & 1    & 0    & 0    & 0    \\
    \cdot &\cdot &\cdot &\cdot &y_{10,5}&y_{10,6} &\cdot\\
    \cdot &\cdot &\cdot &\cdot &y_{11,5}&y_{11,6}&y_{11,7}\\
    \cdot &\cdot &\cdot &\cdot & 1     &\Blue{y_{12,6}}&y_{12,7}\\
    \cdot &\cdot &\cdot &\cdot &\cdot  &1     &y_{13,7}\\
    \cdot &\cdot &\cdot &\cdot &\cdot  &\cdot  &1 
\end{matrix}\right]$}
\end{picture}
\caption{Stiefel coordinates corresponding to a checkerboard.}
\label{F:ChildBoard}
\end{figure}
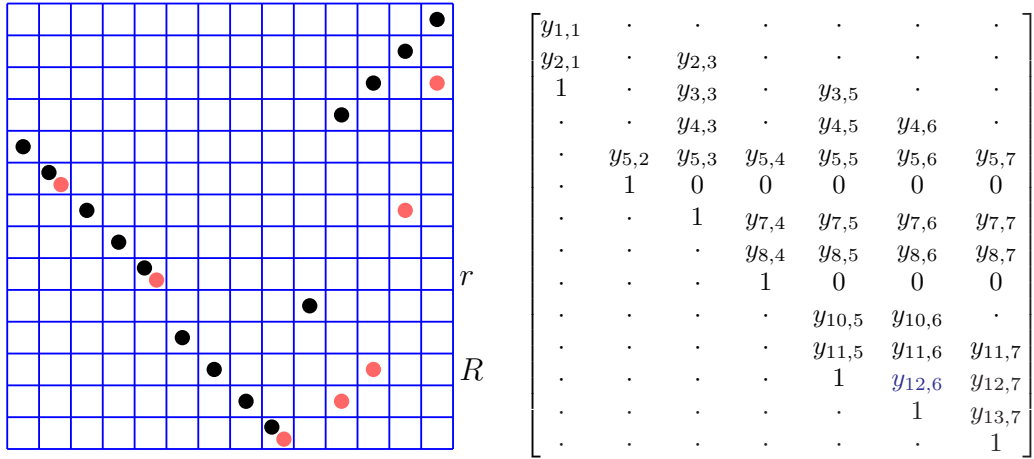
$\cbd'$ of Figure~\ref{F:CheckerBoardVariety} with coordinates $\calY_{\cbd'}$,
where the move connecting them is the swap move in the center of Table~\ref{t:move}.
Comparing these two figures will help to explain our arguments.
In Figure~\ref{F:ChildBoard}, we have $s=4$, the red checker $s$ is to the left in row
$r=9$, and the red checker $s{+}1$ is to the right in row $R=12$. 
These two are in different columns in Figure~\ref{F:CheckerBoardVariety}.

We define $\calY_{\cbd}(t)=(y_{i,j}(t))$.
The entry $y_{i,j}(t)$ will depend on the position of the red checker $j$.
Recall that the black checkers are in regions $A$, $B$, $E$, or in row $r$.
 \begin{enumerate}
  \item If $j\neq s$, set $y_{i,j}(t):=y_{i,j}$.

  \item  When $j=s$, set $y_{r,s}(t):=y_{r+1,s+1}$ and $y_{r+1,s}(t):=-ty_{r+1,s+1}$, and 
 \[
   \begin{array}{rclcl}
      y_{a,s}(t)& :=& -t y_{a,s+1} &\qquad& \mbox{ for }a\in A, \\
      y_{b,s}(t)& :=& y_{r+1,s+1}\cdot y_{b,s} &\qquad& \mbox{ for }b\in B,
   \end{array}
  \] 
    and if $e\in E\smallsetminus\{r{+}1\}$, then $y_{e,s}(t)=0=y_{e,s}$, as $s$ is in row $r<e$.

    The terms $-ty_{a,s{+}1}$ for $a\in A$ occur only if the red checker $s$ in the critical row in
    $\cbd'$ is not in the square of the descending checker.

 \end{enumerate}
Observe that $M\calY_{\cbd}$ is equal to $M\calY_{\cbd}(t)$, except in column $s$, and that if $\bfh_s$ and $\bfh_s(t)$
are the vectors of column $s$ in $M\calY_{\cbd}$ and in $M\calY_{\cbd}(t)$ respectively, then  
 \begin{equation}\label{Eq:bfh_s}
    \bfh_s(t)\ =\  y_{r+1,s+1}\bfh_s \ -\ t\Bigl( y_{r+1,s+1}\bfm_{r+1} +\sum_{a\in A} y_{a,s+1}\bfm_a\Bigr)\,,
 \end{equation}
where the term inside the parentheses is a sum of components of the column vector $\bfh_{s+1}$.

\begin{lemma}\label{L:St_family_III}
 For any $t\neq 0$, $\phi(M\calY_{\cbd}(t))$ is dense in the checkerboard variety $Y_{\cbd'}(F,M'(t))$ and
 $\phi(M\calY_{\cbd}(0))$ is dense in $Y_{\cbd}(F,M'(0))$.
\end{lemma}

\begin{proof}
 Note that $\calY_{\cbd}(0)=\calY_{\cbd}$, except in their $s$th columns.
 These columns are proportional, as $y_{i,s}(0)=y_{r+1,s+1}\cdot y_{i,s}$, for all $i$.
 This proves the statement for $t=0$.

 For $t\neq 0$, we show that $\phi(M\calY_{\cbd}(t))$ is dense in the checkerboard 
 variety $Y_{\cbd'}(F,M'(t))$ by describing Stiefel coordinates
 $\calY_{\cbd'}(t)=(y'_{i,j}(t))$ with 
 $\phi(M\calY_{\cbd}(t))=\phi(M'(t)\calY_{\cbd'}(t))$ that 
 have the following properties:
 \begin{equation}\label{Eq:ZProperties}
  \raisebox{12pt}{\begin{minipage}[t]{400pt}
   The transformation $\calY_{\cbd}(t)\to\calY_{\cbd'}(t)$ is invertible, and 
    the entry $y'_{i,j}(t)$ of $\calY_{\cbd'}(t)$ is $1$, $0$, or a function of 
    the $y_{p,q}$ and $t$ if and only if the entry in the Stiefel coordinates
    $\calY_{\cbd'}$ of Definition~\ref{Def:Stiefel_Coords} is $1$, $0$, or an indeterminate, respectively. 
  \end{minipage}}
 \end{equation}

 Let $\bfh_1(t),\dotsc,\bfh_k(t)$ be the $k$ column vectors of $M\calY_{\cbd}(t)$.
 We use these to define the entries $y'_{i,j}(t)$ of $\calY_{\cbd'}(t)$, which depend upon the position of the
 red checker $j$ in $\cbd'$.
 Recall from Figure~\ref{F:CheckerBoardVariety} that red checkers in $\cbd'$
 lie in one of the regions $A$--$F$. 

 If the red checker $j$ is in a row above $r$, so that  $j\in A$, $B$, or $C$, then 
 \begin{equation}\label{Eq:jinABC}
   \bfh_j(t)\ =\ \sum_{i\in A,B} y_{i,j}\bfm_i
            \ =\ \sum_{i\in A,B} y_{i,j}\bfm'_i(t)\,.
 \end{equation}

If $j=s$, then we have 
 \begin{equation}\label{Eq:j=s}
   \bfh_s(t)\ =\ \sum_{a\in A} -t  y_{a,s+1}\bfm_a\ 
               \ +\ y_{r+1,s+1}\Bigl( \sum_{b\in B} y_{b,s}\bfm_b\ +\ \bfm_r-t\bfm_{r+1}\Bigr)\,.
 \end{equation}

If $j=s{+}1$, so that the red checker is in row $R$, 
 \begin{equation}\label{Eq:j=s+1}
   \bfh_{s+1}(t)\ =\  \sum_{a\in A} y_{a,s+1}\bfm_a\ 
                    \ +\ \ \sum_{b\in B} y_{b,s+1}\bfm_b\
                    \ +\ \ \sum_{e\in E\smallsetminus\{R\}} y_{e,s+1}\bfm_e\ +\ \bfm_{R}\,.
 \end{equation}
When $R=r{+}1$, the last sum is empty, and the last term is $\bfm_{r+1}$.
Also, we always have $y_{r,s+1}=0$ as the red checker $s$ lies in the square of black checker $r$,
which is northwest of red checker $s{+}1$.

For all other red checkers $j$,  either $j\in F$ or 
$j\in E\smallsetminus\{R\}$, and 
$\bfh_{j}(t)\ =\  \sum_{i=1}^n y_{i,j}\bfm_i$.
Note that $y_{r,j}=0$ as red checker $s$ lies in the square of black checker $r$,
and both are northwest of red checker $j$.
For $j\in E\smallsetminus\{R\}$, we have $y_{r+1,j}=0$ as black checker $r{+}1$ is east of red checker
$j$. 

To define $y'_{i,j}(t)$, recall that $\bfm'_r(t)=\bfm_r-t\bfm_{r+1}$, $\bfm'_{r+1}(t)=\bfm_r$, and 
$\bfm'_i(t)=\bfm_i$ 
for $i\neq r, r{+}1$.
If $j\in A$, $B$, or $C$, then by~\eqref{Eq:jinABC}, we may define
$y'_{i,j}(t)=y_{i,j}$, for then 
 \begin{equation}\label{Eq:h'_in_m'}
   \bfh_j(t)\ =\ \sum_{i=1}^n y'_{i,j}(t)\bfm'_i(t)\,.
 \end{equation}
As checkers above row $r$ do not move, the entries $y'_{i,j}(t)$ for these $j$ have the
properties~\eqref{Eq:ZProperties}.

For $j=s$, we rewrite~\eqref{Eq:j=s} in terms of $\bfm'_i(t)$ to get
\[
   \bfh_s(t)\ =\  \sum_{a\in A} -t y_{a,s+1}\bfm_a'(t)\ 
               +\ y_{r+1,s+1}\Bigl( \sum_{b\in B} y_{b,s}\bfm'_b(t)\ +\ \bfm'_r(t)\Bigr)\,.
\]
Define $y'_{r,s}(t)=1$, $y'_{b,s}(t)=y_{b,s}$ for $b\in B$, 
$y'_{a,s}(t)=-t\cdot y_{a,s+1}/y_{r+1,s+1}$ for $a\in A$, and $y'_{i,s}(t)=0$ for $i\in E$.
With these definitions, we have
 \begin{equation}\label{bfh_s_again}
       \bfh_s(t)\ =\ y_{r+1,s+1} \cdot \Bigl( \sum_{i=1}^n y'_{i,s}(t)\bfm'_i(t) \Bigr)\,,
 \end{equation}
so that~\eqref{Eq:h'_in_m'} holds (up to the factor $ y_{r+1,s+1}$) for $j=s$.
Also,~\eqref{Eq:ZProperties} holds as in $\cbd'$ red checker $s$ lies in the same column as 
red checker $s{+}1$ in $\cbd$, and thus below the same black checkers in $A$ as  red checker $s{+}1$.

For $j=s{+}1$, replace $\bfh_{s+1}(t)$ by $\defcolor{\bfh'_{s+1}(t)}:=\bfh_{s+1}(t)+\frac{1}{t}\bfh_s(t)$.
Note that both $\bfh_s(t),\bfh_{s+1}(t)$ and  $\bfh_s(t),\bfh'_{s+1}(t)$ have the same span.
By~\eqref{Eq:j=s} and~\eqref{Eq:j=s+1}, this cancels the sums involving $A$ and the terms involving
$\bfm_{r+1}$. 
Its form is slightly different in the two cases $R>r{+}1$ and $R=r{+}1$.
When $R>r{+}1$, $\bfh'_{s+1}(t)$ becomes 
\[
   \sum_{b\in B} (\tfrac{1}{t}y_{r+1,s+1}\cdot y_{b,s} + y_{b,s+1}) \bfm_b
    \quad +\ \tfrac{1}{t} y_{r+1,s+1} \bfm_r \ +\ 
     \sum_{e\in E\smallsetminus\{r+1,R\}} y_{e,s+1} \bfm_e 
    \quad +\ \bfm_R\,.
\]
When $R=r{+}1$, we have $y_{r+1,s+1}=1$ and $\bfh'_{s+1}(t)$ is
\[
    \sum_{b\in B} (\tfrac{1}{t} y_{b,s} + y_{b,s+1}) \bfm_b \ +\ \tfrac{1}{t}\bfm_r\,.
\]
Let $y'_{r+1,s+1}(t)$ be the coefficient of $\bfm_r=\bfm'_{r+1}(t)$ in these expressions 
and for $i\neq r{+}1$, let $y'_{i,s+1}(t)$ be the coefficient of $\bfm_i=\bfm'_i(t)$.
As $\bfm'_r(t)$ does not appear,~\eqref{Eq:h'_in_m'} holds for $j=s{+}1$,
and these functions $y'_{i,s+1}(t)$ satisfy the properties~\eqref{Eq:ZProperties}.

We illustrate these definitions of $\bfh_s(t), \bfh_{s+1}(t),$ and $\bfh'_s(t)$ for the checkerboard $\cbd$ of
Figure~\ref{F:ChildBoard}. 
    Below are the columns $s$ and $s{+}1$ of the Stiefel
  coordinates $\calY_{\cbd}(t)$, which correspond to the vectors $\bfh_s(t)$ and $\bfh_{s+1}(t)$, and a column
  corresponding to the $\bfh'_{s+1}(t)$.
\begin{equation}\label{Eq:M_matrix}
{\TEN
\left[\begin{matrix}
      \cdot      & \cdot    & \cdot    \\
      \cdot      & \cdot    & \cdot    \\
     -ty_{3,5}    & y_{3,5}    & \cdot    \\
     -ty_{4,5}    & y_{4,5}    & \cdot    \\
   y_{10,5}y_{5,4} & y_{5,5}     & y_{5,5}+\tfrac{1}{t}y_{10,5}y_{5,4} \\
      \cdot      & \cdot    & \cdot      \\
   y_{10,5}y_{7,4} & y_{7,5}     & y_{7,5}+\tfrac{1}{t}y_{10,5}y_{7,4} \\\rule{0pt}{12pt}
   y_{10,5}y_{8,4} & y_{8,5}     & y_{8,5}+\tfrac{1}{t}y_{10,5}y_{8,4} \\
     y_{10,5}     & \cdot     & \tfrac{1}{t}y_{10,5} \\
   -t y_{10,5}     & y_{10,5}   & \cdot   \\
     \cdot      &  y_{11,5}   & y_{11,5} \\
     \cdot      &   1        &    1    \\
     \cdot      & \cdot      & \cdot      \\
     \cdot      & \cdot      & \cdot      \\
    \bfh_s(t)  &  \bfh_{s+1}(t)  &  \bfh'_{s+1}(t)  
\end{matrix}\right]}
\end{equation}

In the remaining cases, $j\in E\smallsetminus\{R\}$ and $j\in F$, the rows of the 0 entries
in those columns of $\calY_{\cbd}$ and $\calY_{\cbd'}$ are different.
For example, in $\calY_{\cbd}$ the entries in row $R$ are indeterminates, while they are 0 in
$\calY_{\cbd'}$. 
This is because the red checker $s{+}1$ in row $R$ is not in the square of the black checker in $\cbd$, but it is in
  that square in $\cbd'$.
This is observed in Figure~\ref{F:ChildBoard}, where the entry $y_{12,6}\neq 0$, but it is zero in Figure~\ref{F:CheckerBoardVariety}.
To obtain this zero entry in $\calY_{\cbd'}(t)$, we use $\bfh'_{s+1}(t)$ to reduce $\bfh_j(t)$.

If $j\in E\smallsetminus\{R\}$, note that $y_{r,j}=0=y'_{r,j}$.
Indeed, in $\cbd$, the red checker $s$ lies in the square of black checker $r$, while in $\cbd'$, the black
checker $r$ is northeast of the red checker $j$. 
Also, $y_{r+1,j}=0$, as the black checker $r{+}1$ is northeast of the red checker $j$ in $\cbd$.
Set $\defcolor{\bfh'_j(t)}:=\bfh_j(t)-y_{R,j}\bfh'_{s+1}(t)$.
When $R=r{+}1$, $y_{R,j}=0$ so $\bfh'_j(t)=\bfh_j(t)$, and otherwise
\[
  \bfh'_j(t)\ =\ \sum_{i\in B,E\smallsetminus\{R\}}(y_{i,j}-y_{R,j}\cdot y'_{i,s+1}(t))\bfm'_i(t)
                  \ -\  y_{R,j}\cdot y'_{r+1,s+1}(t) \bfm'_{r+1}(t)\,.
\]
Let $y'_{i,j}(t)$ be the coefficient of $\bfm'_i(t)$ in this expression.
Since red checker $j$ is in a row $\rho$ below red checker $s{+}1$, $y'_{\rho,s+1}(t)=0$ so
$y'_{\rho,j}(t)=1$. 
Also note that $y'_{r+1,j}(t)=-y_{R,j}\cdot y'_{r+1,s+1}(t)$ and $y'_{R,j}(t)=0$, by construction.

If $j\in F$, then the differences between $\calY_{\cbd}$ and $\calY_{\cbd'}$ are that $y_{r,j}=y'_{R,j}=0$ and both
$y_{R,j}$ and $y'_{r,j}$ are indeterminates. 
We observe this in column six in Figures~\ref{F:CheckerBoardVariety} and~\ref{F:ChildBoard}.
Suppose that $R>r{+}1$.
Then 
\[
   \bfh_j(t)\ =\ \sum_{i\in A,B,E\smallsetminus\{r+1,R\}} y_{i,j} \bfm_i
                \quad+\ y_{r+1,j}\bfm_{r+1}\ +\ y_{R,j}\bfm_R\,.
\]
Set $\defcolor{\bfh'_j(t)}:=\bfh_j(t)-y_{R,j}\bfh'_{s+1}(t)$, which is
\[
  \sum_{a\in A}y_{a,j}\bfm_a\ +\ 
  \sum_{i\in B,E\smallsetminus\{r+1,R\}} (y_{i,j} -y_{R,j}\cdot y'_{i,s+1}(t))\bfm_i
  \ \ -\ \tfrac{1}{t} y_{R,j}y_{r+1,s+1}\bfm_r \ +\ y_{r+1,j}\bfm_{r+1}\,.
\]
To rewrite this in terms of $\bfm_i'(t)$, by~\eqref{Eq:movingFlag} $\bfm_r=\bfm'_{r+1}(t)$,  
$\bfm_{r+1}=\frac{1}{t}(\bfm'_{r+1}(t)-\bfm'_r(t))$, and otherwise $\bfm_i=\bfm'_i(t)$, which gives
 \begin{multline*}
  \qquad  \bfh'_{j}(t)\ =\ 
  \sum_{a\in A}y_{a,j}\bfm'_a(t)\ +\ 
  \sum_{i\in B,E\smallsetminus\{r+1,R\}} (y_{i,j} -y_{R,j}\cdot y'_{i,s+1}(t))\bfm'_i(t) \\
  \ -\ \tfrac{1}{t}y_{r+1,j}\bfm'_r(t) 
   \ +\ \tfrac{1}{t}( y_{r+1,j}- y_{R,j}y_{r+1,s+1})\bfm'_{r+1}(t)\,.\qquad
 \end{multline*}
Let $y'_{i,j}(t)$ be the coefficient of $\bfm'_i(t)$ in this expression.
This expression shows that~\eqref{Eq:ZProperties} holds when $R>r{+}1$.
  
The argument is simpler when $R=r{+}1$, for then 
\[
   \bfh_j(t)\ =\ \sum_{i\in A,B,E\smallsetminus\{r+1\}} y_{i,j} \bfm_i
                \quad+\ y_{r+1,j}\bfm_{r+1}\,.
\]
and so $\defcolor{\bfh'_j(t)}:=\bfh_j(t)-y_{r+1,j}\bfh'_{s+1}(t)$ is 
\[
  \sum_{a\in A} y_{a,j}\bfm_a(t)\ +\ 
  \sum_{i\in B,E\smallsetminus\{r+1\}} (y_{i,j} - y_{r+1,j}\cdot y'_{i,s+1}(t)) \bfm_i(t)\,.
\]
Let $y'_{i,j}(t)$ be the coefficient of $\bfm'_i(t)$ in this expression.
Then
\[
   \bfh'_{j}(t)\ =\ \sum_{i=1}^n y'_{i,j}(t) \bfm'_i(t)\,,
\]
and these functions $y'_{i,j}(t)$ satisfy the properties~\eqref{Eq:ZProperties}.
\end{proof}

\begin{remark}
 In this proof, when $t\neq 0$ and for $j=s{+}1$, $j\in F$, or $j\in E\smallsetminus\{R\}$, we replaced
 $\bfh_j(t)$ by $\bfh'_j(t)=\bfh_j(t)-z\bfh'_\ell(t)$ where $\ell<j$ and $z$ is the coefficient of $\bfm'_i(t)$
 in $\bfh_j(t)$ and  $\bfm'_i(t)$ is the leading term in $\bfh'_\ell(t)$ (with coefficient 1).
 In all these cases, this put the vectors $\bfh_1(t),\dotsc,\bfh_k(t)$ into reduced echelon form
 with respect to the basis $M'(t)$.
 The content of the proof was that the resulting matrix $\calY_{\cbd'}(t)$ of coefficients satisfies  the
 properties~\eqref{Eq:ZProperties}. 
 Our software automatically performs this reduction to change coordinates from $\calY_{\cbd}(1)$ to
 $\calY_{\cbd'}(1)=\calY_{\cbd'}$ for the node $\cbd'$.\hfill$\diamond$
\end{remark}

\begin{remark}
  The formulation in Case III can lead to numerical instability in computation.
  From~\eqref{Eq:bfh_s}, $\bfh_s(t)$ (column $s$ in $M\calY_{\cbd}(t)$) is obtained by
  multiplying $\bfh_s$ (column $s$ in $M\calY_{\cbd}$) by $y_{r+1,s+1}$ and subtracting part of column $s{+}1$ in
  $M\calY_{\cbd}$ multiplied by $t$.
  (See also~\eqref{bfh_s_again} and the first column of the matrix~\eqref{Eq:M_matrix}, where $y_{r+1,s+1}=y_{10,5}$.)
  This leads to numerical instability in a computation when $y_{r+1,s+1}$ is close to zero.\hfill$\diamond$
\end{remark}

\subsection{Littlewood-Richardson Homotopy Algorithm}\label{SS:allTogether}
Using the definitions and results of the previous subsections, including Algorithm~\ref{Alg:core}, we
describe the Littlewood-Richardson Homotopy Algorithm.

Let $F$ be the flag in $\C^n$ corresponding to the identity matrix, and 
let $M$ be the opposite flag.
This corresponds to the permutation array for $\omega_0$ and the matrix \defcolor{$J$} 
with 1s along its anti-diagonal.
These flags are at the root of each checkerboard game.
Fix a Schubert problem $(\beta^1,\dotsc,\beta^s)$ for $\Gr(k,n)$ and consider its checkerboard tournament
$\calT$.
Every node in $\calT$ is a checkerboard $\cbd$ and has an  intermediate Schubert
problem~\eqref{Eq:Intermediate_def}, for flags $F^\ell,\dotsc,F^s$ which will be determined in the algorithm.
The checkerboard game of such a node lies in level $\ell{-}2$ of $\calT$.

\bigskip
\begin{breakablealgorithm}
\caption{Littlewood-Richardson Homotopy Algorithm}\label{Alg:LRH}
\begin{algorithmic}[1]
\REQUIRE {An instance of a Schubert problem in $\Gr(k,n)$ given by
   two positive integers $k<n$, a list of brackets $(\beta^1,\dotsc,\beta^s)$ such that 
   $\|\beta^1\|+\dotsb+\|\beta^s\|=k(n{-}k)$, and flags $E^1,\dotsc,E^s$ represented by invertible $n\times n$ matrices. 
   }
\ENSURE{All solutions to the instance
    \begin{equation}\label{Eq:userInstance}
       X_{\beta^1}E^1\;\cap\; \dotsb\;\cap\; X_{\beta^s} E^s\,,
    \end{equation}
    represented in Stiefel coordinates as $n\times k$ matrices.
   }
\STATE{Generate random upper unitriangular $n\times n$ matrices $A_3,\dotsc,A_s$.}
\STATE{Compute the checkerboard tournament $\calT$ for $\beta^1,\dotsc,\beta^s$.}
\STATE{Populate each node $\cbd$ of $\calT$ with an empty list of solutions and with flags 
   \begin{equation}\label{Eq:flagDefinitions}
       F^\ell\ :=\ A_\ell J\,,\ 
       F^{\ell+1}\ :=\ A_\ell A_{\ell+1} J\,,\  \dotsc\,,\ 
       F^s\ :=\ A_\ell A_{\ell+1}\dotsb A_s J\,,
   \end{equation}
    where $\cbd$ lies in a checkerboard game at level $\ell-2$ of $\calT$, and the corresponding 
    intermediate Schubert problem is~\eqref{Eq:Intermediate_def}. 
    Mark the node as  as {\tt`unresolved'}.
   }
\STATE{
   Populate the leaf of the last checkerboard game with the single solution~\eqref{Eq:trivSChProb} to
   \[
      X_{(\beta^s)^\vee} F \cap X_{\beta^s} A_s J 
   \]
  represented in Stiefel coordinates as the echelon form of the submatrix of $A_s$ consisting of its columns
  $n{+}1{-}\beta^s_i$ for $i=1,\dotsc,k$.
  Mark this node as {\tt`resolved'}.
}
\WHILE{Node $\cbd'$ of $\calT$ is unresolved}
	\IF {any child of $\cbd'$ is unresolved} 
		\STATE {replace $\cbd'$ by this child and return.}
	\ENDIF
	\IF {all children of $\cbd'$ are resolved}
		\FOR{\textbf{each} child $\cbd$ of $\cbd'$}
			\IF{$\cbd'$ is a leaf of a checkerboard game at level $\ell{-}2$}
                         \STATE{$\cbd$  is the root of a game at level $\ell{-}1$.}
				\FORALL{solutions $y=(y_{i,j})$ in node $\cbd$} 
					\STATE{append $A_\ell (y_{i,j})$ to the list of solutions in $\cbd'$.}
				\ENDFOR
			\ELSIF {$\cbd$ is a child of $\cbd'$ in the same checkerboard game as $\cbd'$}
				\FORALL{solutions $y$ of node $\cbd$}
					\STATE{Use Algorithm~\ref{Alg:core} to obtain the corresponding solution $y'$ }
					\STATE{Append $y'$ to the list of solutions for $\cbd'$.}
				\ENDFOR
			\ENDIF
		\ENDFOR
	\ENDIF
\ENDWHILE
\STATE{
   When all nodes of $\calT$ are resolved, the solutions at its root are all solutions to the instance
   \[
      X_{\beta^1} F\;\cap\; X_{\beta^2} J\ \cap\
      X_{\beta^3} (A_3J)\;\cap\;  \dotsb \;\cap\; X_{\beta^s} (A_3A_4\dotsb A_s J)\,.
   \]
    Replace each solution $y$  at the root by $E^1 y$, producing all solutions to the
    instance of this  Schubert problem given by the flags $E^1=E^1F, (E^1J), (E^1A_3J),\dotsc$.
%
%
   }
\STATE{Create a homotopy between these flags and the user-defined flags $E^1,E^2,\dotsc,E^s$ 
   and follow these points $E^1 y$ along that homotopy, to obtain all solutions to the user's 
   instance~\eqref{Eq:userInstance}.
   }
\end{algorithmic}
\end{breakablealgorithm}

\begin{proof}[Proof of correctness.]
  We prove that the algorithm performs as described when the input flags $E^1,\dotsc,E^s$ are general.
  This will also prove Theorem~\ref{Th:Main}.
  Every node $\cbd'$ in the checkerboard tournament corresponds to an intermediate Schubert problem
 \begin{equation}\label{Eq:ISP}
   Y_{\cbd'}(F,M')\ \cap\ X_{\beta^\ell}F^{\ell}\ \cap\ \dotsb\ \cap\ X_{\beta^s}F^s\,,
 \end{equation}
 where $\cbd'$ is a node in a checkerboard game at level $\ell-2$ in $\calT$ and the flags $F^\ell,\dotsc,F^s$ are as
 defined by~\eqref{Eq:flagDefinitions}.
 Let \defcolor{$S(\cbd')$} be the set of solutions to this intermediate Schubert problem~\eqref{Eq:ISP}.
 We claim that, when a node $\cbd'$ is resolved in Algorithm~\ref{Alg:LRH}, the set of solutions in that node (as
 constructed in Steps 10--22) equals $S(\cbd')$, recorded in the Stiefel coordinates $\calY_{\cbd'}$ of
 Subsection~\ref{SS:CBV}. 
 Establishing this claim, as well as the arguments presented below about Steps 25 and 26, will complete the proof of
 correctness of Algorithm~\ref{Alg:LRH}.

 For any checkerboard $\cbd$, the Stiefel coordinates $\calY_{\cbd}$ parameterize only a dense
   subset of a checkerboard  variety $Y_\cbd(F,M)$.
   Our arguments below ignore this distinction.
   To validate them, note that for each checkerboard the points of the checkerboard  variety $Y_\cbd(F,M)$
   that are not parameterized by $\calY_{\cbd}$ form a proper subset, $Z$.
   As the flags $F^i$ are general, Kleiman's Theorem~\cite{Kl74} asserts that there will be no points of~\eqref{Eq:ISP}
   that lie in $Z$. 
   As there are only finitely many checkerboards, the choice of general flags $F^i$ and $E^i$ will guarantee that the
   algorithm computes all solutions to~\eqref{Eq:userInstance}.

 We prove the claim by induction on $\calT$.
 The claim holds at the leaf of $\calT$, by construction:
 Step 4 places the unique solution of the intermediate problem of the leaf (explained at the end of
 Subsection~\ref{SS:SchubertProblems}), and marks that node as resolved.\smallskip

 Suppose that $\cbd'$ is a node of $\calT$ that is not the leaf of any checkerboard game in $\calT$. 
 Then either $\cbd'$ has a unique child $\cbd$ or possibly two, $\cbd$ and $\cbd''$, in that checkerboard game.
 Before node $\cbd'$ is resolved, its child node(s) must be resolved.
 By the induction hypothesis, Algorithm~\ref{Alg:LRH} has populated $\cbd$ with the solutions $S(\cbd)$ to its
 intermediate problem, and the same for $\cbd''$ if it exists.
 The points $S'(\cbd')$ used by Algorithm~\ref{Alg:LRH} to populate the node $\cbd'$ are obtained
 from the solutions in
 $S(\cbd)$ (and $S(\cbd'')$) using Algorithm~\ref{Alg:core}, which follows them along the homotopy induced by the family
 $M\calY_{\cbd}(t)$ (or $M\calY_{\cbd''}(t)$).

 In the geometric Littlewood-Richardson rule, these families are Stiefel coordinates for the family
 $Y_{\cbd,\cbd'}(t)$ for $t\in \C$ with fiber $Y_{\cbd'}(F,M')$ over $t=1$ and fiber 
 $Y_{\cbd}(F,M)$ (or  $Y_{\cbd}(F,M)\cup Y_{\cbd''}(F,M)$) over $t=0$.
 This implies that $S'(\cbd')$ is the set of solutions to the intermediate problem at the node $\cbd'$.

 We prove the claim when $\cbd'$ is a leaf of a checkerboard game.
 Such a leaf has only one child in the tournament $\calT$, which is the root of the subsequent checkerboard game.
 In this case, there is a bracket $\gamma$ such that the intermediate problems at these two nodes are
 \begin{eqnarray}
  \makebox[0.6cm][l]{$\cbd'$}
     &&  X_\gamma F\;\cap\; X_{\beta^\ell}F^{\ell}\;\cap\;\dotsb\;\cap\; X_{\beta_s} F^s  \label{Eq:cbdp}\\
  \makebox[0.6cm][l]{$\cbd$}
   &&  X_\gamma F\;\cap\; X_{\beta^\ell}J\;\cap\; 
         X_{\beta^{\ell+1}}\widetilde{F}^{\ell+1}\;\cap\;\dotsb\;\cap\; X_{\beta_s}\widetilde{F}^s\,,  \label{Eq:cbd}
 \end{eqnarray}
 where the flags $F^\ell,\dotsc,F^s$ are defined by~\eqref{Eq:flagDefinitions}, as are the flags
 $\widetilde{F}^{\ell+1},\dotsc,\widetilde{F}^s$, except that the index $\ell$ of the ambient checkerboard poset
 changes, so that 
\[
   \widetilde{F}^{\ell+1}\ =\ A_{\ell+1} J\,,\ 
   \widetilde{F}^{\ell+2}\ =\ A_{\ell+1}A_{\ell+2} J\,,\  \dotsc\ ,\ 
   \widetilde{F}^{s}\ =\ A_{\ell+1}\dotsb A_{s} J\,.
\]
 Since $A_\ell F$ and $F$ give the same flag, the intersection~\eqref{Eq:cbdp} is obtained from that of~\eqref{Eq:cbd}
 through left multiplication by $A_\ell$.
 Thus if $\cbd$ is resolved and populated by the points $S(\cbd)$ in the intersection~\eqref{Eq:cbd}, then Steps 13--15
 of Algorithm~\ref{Alg:LRH} populate node $\cbd'$ with all the points $S(\cbd')$ in the intersection~\eqref{Eq:cbdp},
 completing the proof of the claim.

 The argument for Step 25, going from the root of $\calT$, is that the intermediate Schubert problem passes from
\[
    X_{\beta_1} F\;\cap\;
    X_{\beta_2} J\;\cap\;
    X_{\beta_3} F^3\;\cap\ \dotsb\ \cap\;  X_{\beta_s} F^s
\]
 to
 \begin{equation}\label{Eq:penultimate}
    X_{\beta_1} E^1\;\cap\;
    X_{\beta_2} E^1 J\;\cap\;
    X_{\beta_3} E^1 F^3\;\cap\ \dotsb\ \cap\;  X_{\beta_s} E^1 F^s
  \end{equation}
 which is the same as passing between leafs and roots in the proof of the claim.
 Finally, Step 26 is simply applying a parameter homotopy~\cite{LSY89,MS89} between the solutions to~\eqref{Eq:penultimate}
 and those of the original Schubert problem
\[
    X_{\beta_1} E^1\;\cap\;
    X_{\beta_2} E^2\;\cap\;
    X_{\beta_3} E^3\;\cap\ \dotsb\ \cap\;  X_{\beta_s} E^s\,.
\]
 This completes the proof of correctness.
\end{proof}

\section{The Performance of the Implementation}\label{S:Examples}
The Littlewood-Richardson homotopy algorithm has two implementations.
One is in the interpreted language of {\tt Macaulay2}~\cite{M2} 
using its {\tt NumericalAlgebraicGeometry} package~\cite{NAG4M2},
and the other is compiled code and uses the Polynomial Homotopy Continuation
package {\tt PHCpack}~\cite{PHCpack}.
These implementations, as well as implementations of the 
Pieri Homotopy algorithm~\cite{HSS98,HV2000} may be
called from the {\tt NumericalSchubertCalculus} package of {\tt Macaulay2}.
An introduction to its capabilities and use is given in~\cite{NSC_article}.
This software is free and open source, available on {\tt github} with the
compiled version accessible to the Python programmer via {\tt phcpy}~\cite{Ver2014}.

Table~\ref{Ta:timings} gives a selection of the Schubert problems 
this software is able to solve.
\begin{table}[htb]
 \[
 \begin{array}{|l|l|r|r|r|}\hline
   \mbox{Grassmannian} & \mbox{Schubert Problem} & d & 
   \mbox{Interpreted} & \mbox{Compiled} \RP\\ \hline \hline
    \Gr(2,7) &  [5,7]^{10}     &    42 &  249.58 &     1.3652 \RP\\ \hline
    \Gr(2,8) &  [5,8]^6       &    15  &  104.04  &   0.7135\RP\\ \hline
    \Gr(2,9) &  [6,9]^7       &    36  &  455.93   &  3.4541 \RP\\ \hline
    \Gr(2,10)&  [7,10]^7      &   91   & 1899.54   & 17.3442 \RP\\ \hline\hline
%
     \Gr(3,6)  & [3,5,6]^9          &  42 &  148.65 &   1.6758 \RP\\ \hline
     \Gr(3,7) &  [4,6,7]^{10}[3,6,7] &  252 &  2040.51  &   28.5882 \RP\\ \hline
     \Gr(3,8) &  [4,6,8]^5           & 32   &   140.64   &   8.1716 \RP\\ \hline\hline
    \Gr(4,8) &  [3,4,7,8]^4            &   6 &    29.79&    5.7789  \RP\\ \hline
    \Gr(4,8) &  [3,6,7,8]^6[3,4,7,8]  &   50 &   637.15  &   27.4836  \RP\\ \hline
    \Gr(4,8) &  [4,6,7,8]^8[3,4,7,8]^2 & 220 &  3736.61  &   55.8480   \RP\\ \hline
 \end{array}
\]
\caption{Timings of Schubert problems}\label{Ta:timings}
\end{table}
These timings (in seconds) compare the performance of the two implementations of Algorithm~\ref{Alg:LRH}
on the same random instance of the problem.
These were computed on a Macbook Air with a dual-core Intel Core i5 1.6GHz processor.
Here,  the exponents indicate repeated brackets.

The compiled implementation is both faster and more robust.
Table~\ref{Ta:PHC} shows some Schubert problems it can compute, and their timings in h:m:s format.
\begin{table}[htb]
 \[
 \begin{array}{|l|l|r|r|r|}\hline
   \mbox{Grassmannian} & \mbox{Schubert Problem} & d & \mbox{Time} \RP\\ \hline \hline
   \Gr(3,9) & [6,8,9]^{14}[5,8,9]^2    & 30459 &  $59:11:50$    \RP\\ \hline 
   \Gr(4,8) & [4,6,7,8]^{16}           & 24024  & $34:09:46$  \RP\\ \hline 
   \Gr(4,9) & [5,7,8,9]^8 [4,6,8,9]^4 & 25142 &  $293:02:54$    \RP\\ \hline 
   \Gr(5,10) &[4,6,8,9,10]^5 [3,6,7,9,10]^2 & 8860 &  $216:03:54$  \RP\\ \hline  
 \end{array}
\]
\caption{Timings of Schubert problems}\label{Ta:PHC}
\end{table}
These were computed on a single processor of a server with four Six-Core AMD Opteron(tm) 8435 processors, each with an
800MHz clock speed, and 64GB memory.

\bibliographystyle{amsplain}
\bibliography{bibl}

\providecommand{\bysame}{\leavevmode\hbox to3em{\hrulefill}\thinspace}
\providecommand{\MR}{\relax\ifhmode\unskip\space\fi MR }
\providecommand{\MRhref}[2]{%
  \href{http://www.ams.org/mathscinet-getitem?mr=#1}{#2}
}
\providecommand{\href}[2]{#2}
\begin{thebibliography}{10}

\bibitem{BCT}
G.~Bresler, D.~Cartwright, and D.~Tse, \emph{Feasibility of interference
  alignment for the {MIMO} interference channel}, IEEE Trans. Inform. Theory
  \textbf{60} (2014), no.~9, 5573--5586.

\bibitem{By89}
C.I. Byrnes, \emph{Pole assignment by output feedback}, Three Decades of
  Mathematical Systems Theory (H.~Nijmeijer and J.~M. Schumacher, eds.),
  Lecture Notes in Control and Inform. Sci., vol. 135, Springer-Verlag, Berlin,
  1989, pp.~31--78.

\bibitem{EG2002}
A.~Eremenko and A.~Gabrielov, \emph{Pole placement by static output feedback
  for generic linear systems}, SIAM Journal on Control and Optimization
  \textbf{41} (2002), no.~1, 303--312.

\bibitem{Fu97}
W.~Fulton, \emph{Young tableaux}, London Mathematical Society Student Texts,
  vol.~35, Cambridge University Press, Cambridge, 1997.

\bibitem{FRSC-Secant}
L.D. Garc\'ia-Puente, N.~Hein, C.~Hillar, A.~Mart\'in~del Campo, J.~Ruffo,
  F.~Sottile, and Z.~Teitler, \emph{The {S}ecant conjecture in the real
  {S}chubert calculus}, Experimental Mathematics \textbf{21} (2012), no.~3,
  252--265.

\bibitem{M2}
D.R. Grayson and M.E. Stillman, \emph{Macaulay2, a software system for research
  in algebraic geometry}, Available at {\tt
  http://www.math.uiuc.edu/Macaulay2/}.

\bibitem{HaHS}
J.D. Hauenstein, N.~Hein, and F.~Sottile, \emph{A primal-dual formulation for
  certifiable computations in {S}chubert calculus}, Found. Comput. Math.
  \textbf{16} (2016), no.~4, 941--963.

\bibitem{HSDSV}
J.D. Hauenstein, M.~Safey El~Din, {\'{E}}.~Schost, and T.X. Vu, \emph{Solving
  determinantal systems using homotopy techniques}, {\tt arXiv:1802.10409},
  2018.

\bibitem{HHS}
N.~Hein, C.~Hillar, and F.~Sottile, \emph{Lower bounds in real {S}chubert
  calculus}, S{\~a}o Paulo Journal of Mathematics \textbf{7} (2013), no.~1,
  33--58.

\bibitem{HS_Sq}
N.~Hein and F.~Sottile, \emph{A lifted square formulation for certifiable
  {S}chubert calculus}, Journal of Symbolic Compututation \textbf{79} (2017),
  no.~part 3, 594--608.

\bibitem{HSS98}
B.~Huber, F.~Sottile, and B.~Sturmfels, \emph{Numerical {S}chubert calculus},
  Journal of Symbolic Computation \textbf{26} (1998), no.~6, 767 -- 788.

\bibitem{HV2000}
B.~Huber and J.~Verschelde, \emph{Pieri homotopies for problems in enumerative
  geometry applied to pole placement in linear systems control}, SIAM J.
  Control and Optimization \textbf{38} (2000), no.~4, 1265--1287.

\bibitem{KBKK}
S.W. Kim, C.J. Boo, S.~Kim, and H.-C. Kim, \emph{Stable controller design of
  {MIMO} systems in real {G}rassmann space}, International J. of Control,
  Automation and Systems \textbf{10} (2012), no.~2, 213--226.

\bibitem{Kl74}
S.L. Kleiman, \emph{The transversality of a general translate}, Compositio
  Mathematica \textbf{28} (1974), 287--297.

\bibitem{KL72}
S.L. Kleiman and D.~Laksov, \emph{Schubert calculus}, Amer. Math. Monthly
  \textbf{79} (1972), no.~10, 1061--1082.

\bibitem{NAG4M2}
A.~Leykin, \emph{{N}umerical {A}lgebraic {G}eometry}, The Journal of Software
  for Algebra and Geometry \textbf{3} (2011), 5--10, Available at {\tt
  http://j-sag.org/volume3.html}.

\bibitem{NSC_article}
A.~Leykin, A.~Mart\'in~del Campo, F.~Sottile, R.~Vakil, and J.~Verschelde,
  \emph{Software for {N}umerical {S}chubert {C}alculus}, in preparation, 2020.

\bibitem{LS09}
A.~Leykin and F.~Sottile, \emph{Galois groups of {S}chubert problems via
  homotopy computation}, Mathematics of Computation \textbf{78} (2009),
  no.~267, 1749--1765.

\bibitem{LSY89}
T.~Y. Li, Tim Sauer, and J.~A. Yorke, \emph{The cheater's homotopy: an
  efficient procedure for solving systems of polynomial equations}, SIAM
  Journal on Numerical Analysis \textbf{26} (1989), no.~5, 1241--1251.

\bibitem{LWW02}
T.Y. Li, X.~Wang, and M.~Wu, \emph{Numerical {S}chubert calculus by the {P}ieri
  homotopy algorithm}, SIAM Journal on Numerical Analysis \textbf{40} (2002),
  no.~2, 578--600.

\bibitem{GIVIX}
A.~Mart\'\i n~del Campo, F.~Sottile, and R.~Williams, \emph{Classification of
  {S}chubert {G}alois groups in {G}r(4,9)}, {\tt arXiv.org/1902.06809}, 2019.

\bibitem{MS}
A.~Mart{\'{\i}}n~del Campo and F.~Sottile, \emph{Experimentation in the
  {S}chubert calculus}, Schubert Calculus---Osaka 2012 (H.~Naruse, T.~Ikeda,
  M.~Masuda, and T.~Tanisaki, eds.), Advanced Studies in Pure Mathematics,
  vol.~71, Mathematical Society of Japan, 2016.

\bibitem{Mor87}
A.~Morgan, \emph{Solving polynomial systems using continuation for engineering
  and scientific problems}, Classics in Applied Mathematics, vol.~57, SIAM,
  2009.

\bibitem{MS89}
A.~P. Morgan and A.~J. Sommese, \emph{Coefficient-parameter polynomial
  continuation}, Appl. Math. Comput. \textbf{29} (1989), no.~2, part II,
  123--160.

\bibitem{So97}
F.~Sottile, \emph{Pieri's formula via explicit rational equivalence}, Canadian
  Journal of Mathematics. Journal Canadien de Math\'ematiques \textbf{49}
  (1997), no.~6, 1281--1298.

\bibitem{So_Shap}
\bysame, \emph{Real {S}chubert calculus: polynomial systems and a conjecture of
  {S}hapiro and {S}hapiro}, Experimental Mathematics \textbf{9} (2000), no.~2,
  161--182.

\bibitem{RSEG}
\bysame, \emph{Real solutions to equations from geometry}, University Lecture
  Series, vol.~57, American Mathematical Society, Providence, RI, 2011.

\bibitem{SVV}
F.~Sottile, R.~Vakil, and J.~Verschelde, \emph{Solving {S}chubert problems with
  {L}ittlewood-{R}ichardson homotopies}, I{SSAC} 2010---{P}roceedings of the
  2010 {I}nternational {S}ymposium on {S}ymbolic and {A}lgebraic {C}omputation,
  ACM, New York, 2010, pp.~179--186.

\bibitem{SW}
F.~Sottile and J.~White, \emph{Double transitivity of {G}alois groups in
  {S}chubert calculus of {G}rassmannians}, Algebraic Geometry \textbf{2}
  (2015), no.~4, 422--445.

\bibitem{Va06a}
R.~Vakil, \emph{A geometric {L}ittlewood-{R}ichardson rule}, Annals of
  Mathematics. Second Series \textbf{164} (2006), no.~2, 371--421, Appendix A
  written with A. Knutson.

\bibitem{Vak06b}
\bysame, \emph{Schubert induction}, Annals of Mathematics. Second Series
  \textbf{164} (2006), no.~2, 489--512.

\bibitem{PHCpack}
J.~Verschelde, \emph{Algorithm 795: {PHC}pack: A general-purpose solver for
  polynomial systems by homotopy continuation}, ACM Transactions on
  Mathematical Software \textbf{25} (1999), no.~2, 251--276.

\bibitem{Ver2000}
\bysame, \emph{Numerical evidence for a conjecture in real algebraic geometry},
  Experimental Mathematics \textbf{9} (2000), no.~2, 183--196.

\bibitem{Ver2014}
\bysame, \emph{Moderizing {PHC}pack through phcpy}, Proceedings of the 6th
  European Conference on Python in Science (EuroSciPy 2013) (P.~de~Buyl and
  N.~Varoquaux, eds.), 2014, pp.~71--76.

\bibitem{VW04b}
J.~Verschelde and Y.~Wang, \emph{Computing feedback laws for linear systems
  with a parallel {P}ieri homotopy}, Proceedings of the 2004 International
  Conference on Parallel Processing Workshops, ICPPW'04, 2004, pp.~222--229.

\end{thebibliography}

\end{document}